\documentclass[12pt,a4paper]{article}
\usepackage[utf8]{inputenc}
\usepackage[T1]{fontenc}
\usepackage{amsmath}
\usepackage{amsfonts}
\usepackage{amssymb}
\usepackage{makeidx}
\usepackage{graphicx}
\usepackage[T1]{fontenc}
\usepackage{csquotes}
\usepackage{hyperref}
\usepackage[french,english]{babel}
\usepackage{amsfonts,amsmath,amssymb}
\usepackage{amsthm}
\usepackage{bbold}
\usepackage{cases}
\usepackage{stackengine}
\usepackage{scalerel}
\usepackage{enumitem}
\usepackage{fullpage} 
\usepackage{graphicx}
\usepackage{xcolor}

\newtheorem{thm}{Theorem}[section]    
\newtheorem{prop}[thm]{Proposition}    
\newtheorem{defi}{Definition}[section]                             
\newtheorem{cor}[thm]{Corollary}    
\newtheorem{lemme}[thm]{Lemma}   
\newtheorem{NB}{\textbf{\underline{Remark}}}

\newcommand{\R}{\mathbb R}
\newcommand{\Ge}{\mathbb G}
\newcommand\hh{\mathbb{H}}
\newcommand\bbb{\mathbb{B}}
\newcommand\tbbb{\tilde{\mathbb{B}}}
\newcommand{\pr}{\mathbb{P}}
\newcommand{\esp}{\mathbb{E}}

\newcommand{\bbr}{B^{br}}
\newcommand{\tbbr}{\tilde{B}^{br}}

\newcommand{\ta}{\tilde{\alpha}}
\newcommand{\tg}{\tilde{g}}

\newcommand{\ph}{\varphi}

\newcommand{\Bet}{\boldsymbol{\beta}}
\newcommand{\tX}{\tilde{X}}
\newcommand{\tx}{\tilde{x}}
\newcommand{\tz}{\tilde{z}}
\newcommand{\hx}{\hat{x}}
\newcommand{\hz}{\hat{z}}
\newcommand{\txi}{\tilde{\xi}}
\newcommand{\uX}{\underline{X}}
\newcommand{\utX}{\tilde{\underline{X}}}

\newcommand{\ubbb}{\underline{\mathbb{B}}}
\newcommand{\utbbb}{\tilde{\underline{\mathbb{B}}}}

\newcommand{\uzeta}{\underline{\zeta}}

\newcommand{\bd}{\boldsymbol{\delta}}

\newcommand{\symp}{\odot}
\newcommand{\ve}{v}

\newcommand{\bX}{\bar{X}}
\newcommand{\bZ}{\bar{Z}}

\DeclareMathOperator{\Tr}{Tr}
\DeclareMathOperator{\Span}{Span}
\DeclareMathOperator{\diag}{Diag}
\DeclareMathOperator{\osc}{osc}

\DeclareMathOperator{\trans}{trans}
\usepackage{authblk}

\newcommand{\Hm}[1]{\leavevmode{\marginpar{\tiny%
$\hbox to 0mm{\hspace*{-0.5mm}$\leftarrow$\hss}%
\vcenter{\vrule depth 0.1mm height 0.1mm width \the\marginparwidth}%
\hbox to 0mm{\hss$\rightarrow$\hspace*{-0.5mm}}$\\\relax\raggedright
#1}}}

 \date{}                     

\title{Non co-adapted couplings of {Brownian} motions on free, step 2 Carnot groups}

\author[1]{Magalie Bénéfice}
\affil[1]{Université de Lorraine, CNRS, IECL, F-54000 Nancy, France}
\begin{document}

\maketitle

\begin{abstract}
On the free, step $2$ Carnot groups of rank $n$ $\Ge_n$, the subRiemannian Brownian motion consists in a $\mathbb{R}^n$-Brownian motion together with its $\frac{n(n-1)}{2}$ Lévy areas. In this article we construct an explicit successful non co-adapted coupling of Brownian motions on $\Ge_n$. We use this construction to obtain gradient inequalities for the heat semi-group on all the homogeneous step $2$ Carnot groups. Comparing the first coupling time and the first exit time from a domain, we also obtain gradient inequalities for harmonic functions on $\Ge_n$. These results generalize the coupling strategy by Banerjee, Gordina and Mariano on the Heisenberg group. 
\end{abstract}
	\section{Introduction}
	\subsection{Motivation}
    	Aside from providing a better understanding of the geometry of the state space, the coupling method for Brownian motions is a great tool for many analysis results involving the harmonic functions and the heat semi-group such as Harnack, Poincaré, Sobolev or Wasserstein inequalities (see~\cite{kuwada,WangBInequalities,CranstonSoblogRn,CranstonRefl} for some examples).
	This method has been studied these last decades in the case of Riemannian manifolds ({see for example \cite{Kendall86,pascu2018couplings}}). The case of subRiemannian manifolds is a current topic of interest and have been investigated on the Heisenberg group in~\cite{CranstonKolmogorov,Kolmogorov,kendall-coupling-gnl,kendall2009brownian,kendall2007coupling,banerjee2017coupling,bonnefont2018couplings,CoAdaptSuccessHeisenberg2,luo2024nonmarkovian}, on $SU(2)$ \cite{KendallSU(2),Nonco-adaptedSU(2),luo2024nonmarkovian} and on $SL(2,\mathbb{R})$ \cite{Nonco-adaptedSU(2),luo2024nonmarkovian}. On the spaces considered in this work, the subRiemannian Brownian motion can be written under the form $(X_t,z_t)_t$ where $(X_t)_t$ is a Brownian motion on a Riemannian base and $z_t$ can be interpreted as an area swept by $(X_s)_{s\leq t}$. The usual strategy of coupling consists in defining a coupling $(X_t,\tX_t)_t$ on the Riemannian base. The coupling approach is quite useful in these cases as it can deal with some of the above mentioned problems without the intervention of some geometric or analytic objects that are difficult to define (like the Ricci curvature in subRiemannian manifolds).
	
	The free, step $2$ Carnot groups $\Ge_n$, $n\geq 2$, are interesting to study as they can be projected to all homogeneous step $2$ Carnot groups.
	In this work we consider couplings of the paths of two Brownian motions $(\bbb_t)_t$ and $(\tbbb_t)_t$ on $\Ge_n$ starting at $g$ and $\tg$ respectively. In other words we study the joint law of $(\bbb_t,\tbbb_t)_t$. We are particularly interested in successful couplings, i.e., couplings such that the first meeting time $\tau:=\inf\{t\geq0\ | \ \bbb_t=\tbbb_t\}$ is a.s. finite.
	
	 Successful couplings have been first used to obtain estimates of the Total Variation distance between the laws of the processes. {Indeed, denoting by $P_t$ the associated heat-semi group and constructing the coupling $(\bbb_t,\tbbb_t)_t$ such that $\bbb_t=\tbbb_t$ for all $t\geq \tau$, for any measurable function $f$, it is clear that: \begin{equation}\label{eq: Chap1SemigpIn0}
 	|P_tf(g)-P_tf(\tg)|\leq\esp[|f(\bbb_t)-f(\tbbb_t)|\mathbb{1}_{\tau>t}].
 \end{equation}
 From $\eqref{eq: Chap1SemigpIn0}$ we can obtain the Coupling inequality (or Aldous inequality) for Markov processes  (see~\cite{asmussen2003applied}, chapter VII):}
 \begin{equation}\label{eq: Chap1Aldous}
 	d_{TV}\left(\mathcal{L}(\bbb_t),\mathcal{L}(\tbbb_t)\right)\leq \pr(\tau>t) \text{ for all $t>0$}.
 \end{equation}
 If Inequality (\ref{eq: Chap1Aldous}) can been changed into an equality, the coupling is called maximal (note that a maximal coupling is not always successful). When the two quantities in (\ref{eq: Chap1Aldous}) have the same order in $t$, the coupling is called efficient. In an Euclidean space and more generally for a Riemannian manifold having a kind of "reflection structure" just like the plane or the sphere, a maximal and successful coupling is the reflection coupling (see~\cite{ReflKuwada,HsuSturmMaxEuc}).
 Maximal couplings of càdlàg processes can always be constructed on Polish spaces~\cite{MaximalCouplingSverchkov}. However, most of these constructions are not Markovian and even not co-adapted, i.e., we need some knowledge of the future of one of the process to construct the second one. Their study can then be difficult. 
 
 In the case of SubRiemannian Brownian motions, the recent results on the Heisenberg group $\hh$ suggest that maximal couplings, or at least efficient couplings, can be obtained only with non co-adapted couplings. Indeed, in \cite{banerjee2017coupling}, Banerjee, Gordina and Mariano gave the explicit construction of a non co-adapted successful coupling on $\hh$. They proved not only that this coupling is efficient, but also that co-adapted coupling strategies cannot produce any efficient couplings. In particular, the successful co-adapted strategies on $\hh$ by Ben Arous, Cranston and Kendall \cite{CranstonKolmogorov} and Kendall \cite{kendall2007coupling,kendall-coupling-gnl} are not efficient. To enforce this idea, similar results have been obtained by Banerjee and Kendall for the Kolmogorov diffusion \cite{BanerjeeKolmogorov} which, as the Heisenberg Brownian motion, is an hypoelliptic diffusion. 
 
 Following these works, successful and efficient couplings have been studied in various subRiemannian manifolds. In \cite{KendallSU(2)}, we extended Kendall's coupling from $\hh$ to $SU(2)$. In \cite{Nonco-adaptedSU(2)} we constructed a non co-adapted successful coupling, inspired from from the work by Banerjee, Gordina and Mariano, on $SU(2)$ and on $SL(2,\mathbb{R})$ (on $SL(2,\mathbb{R})$, the strategy is a.s. successful only if the Brownian motions start from the same vertical fiber). Using a different strategy, Luo and Neel \cite{luo2024nonmarkovian} constructed a successful coupling on $\hh$, $SU(2)$ and $SL(2,\mathbb{R})$ (with the same initial conditions). They also extended this method to the non-isotropic Heisenberg groups and the universal covering of $SL(2,\mathbb{R})$. {In particular, when the Brownian motions start from the same vertical fiber, their strategy is not merely efficient but is maximal.}
 
For all these non co-adapted strategies, the coupling rate $\pr(\tau>t)$ can be {upper-bounded by a function of $t$} and of the initial distance $d_{cc}(g,\tg)$ (where $d_{cc}$ is the usual subRiemannian distance, namely the Carnot-Carathéodory distance). {Using \eqref{eq: Chap1SemigpIn0}, this provides an estimate of the Total Variation distance between the laws of the subRiemannian Brownian
motions as well as a direct inequality for the associated heat-semi group:
 \begin{equation}\label{eq: Chap1SemigpIn}
 	|P_tf(g)-P_tf(\tg)|\leq 2\|f\|_{\infty}\pr(\tau>t)
 \end{equation}
 where $f$ is a bounded measurable function.} In particular this leads to estimates of horizontal gradient of the semi group.

In \cite{banerjee2017coupling}, Banerjee, Gordina and Mariano also used the non co-adapted coupling to obtain gradients inequalities of {harmonic functions on bounded domains in the Heisenberg group}. This generalizes an idea developped by Cranston in \cite{CranstonSoblogRn} for Brownian motions on $\mathbb{R}^n$ and in \cite{CranstonRefl} for Brownian motions on the complete Riemannian manifolds with Ricci curvature admitting some lower bound. The main idea is to compare the first coupling time $\tau$ with the first exit times $\tau_D(\bbb)$ and $\tau_D(\tbbb)$ of a domain $D$ for the processes $(\bbb_t)_t$ and $(\tbbb_t)_t$ respectively. On the Heisenberg group, Banerjee, Gordina and Mariano obtained the existence of a constant $C>0$ such that, for any domain $D$, $g\in D$ and $f\in\mathcal{C}(\bar{D})$:
\begin{equation}\label{eq: Chap1 GradientDomainHeisenberg}
 \|\nabla_{\mathcal{H}}f(g)\|_{\mathcal{H}}\leq C\left(1+\frac{1}{\delta_g}+\frac{1+(1+\delta_g)^3}{\delta_g^4}\right) f(g)
 \end{equation}
 where $\delta_g:=\bd(g,D^c)$ and $\bd$ is a pseudo-distance on $\hh$ equivalent to the Carnot-Carathéodory distance.
This also lead to the Cheng-Yau inequality on $\hh$.
 
 \subsection{Main results}\label{subsec: Chap1 ResultatsNonCo-adapteCarnot}
 The main idea of this paper is the generalisation of the results by Banarjee, Gordina and Mariano \cite{banerjee2017coupling} to the free, step $2$ Carnot groups $\Ge_n$ of rank $n\geq 3$.
 Denote by $\mathfrak{so}(n)$ the set of real skew-symmetric matrices of dimension $n\times n$. We define the symplectic form $x\symp {\tx}:=x\cdot {\tx}^{t}-{\tx}\cdot x^{t}$ where $x^t\in M_{1,n}(\R)$ denotes the transpose of $x$.
 
 \begin{defi}
  Set $n\geq 2$. The free step $2$ homogeneous Carnot group of rank $n$ is the Lie group $\Ge_n:=\left(\mathbb{R}^n\times\mathfrak{so}(n),\star\right)$ where:
  \begin{equation*}
  	\left(x,z\right)\star \left({\tx},{\tz}\right)=\left(x+{\tx},z+{\tz}+\frac{1}{2}x\symp {\tx}\right)\text{ for all }\left(x,z\right),\left({\tx},{\tz}\right)\in\mathbb{R}^n\times\mathfrak{so}(n).
  \end{equation*}
\end{defi}

On $\Ge_n$ the Brownian motion can be written $(\bbb_t)_t=(X_t,z_t)_t$ where:\begin{itemize}
    \item $(X_t)_t$ a Brownian motion on $\mathbb{R}^n$;
    \item for $1\leq i,j\leq n$ such that $i\neq j$, $(z_t^{i,j})_t$ is the Levy area between $(X_s^i)_{s\leq t}$ and $(X_s^j)_{s\leq t}$.
\end{itemize}
For $n=2$, $\Ge_n$ is isomorphic to the usual Heisenberg group $\hh$. 

Our first result is the construction of an explicit non co-adapted coupling  on $G_n$:
\begin{thm}\label{thm: successfulGn}
    Set $n\geq 2$. Let $g=(x,z)$, $\tilde{g}=(\tilde{x},\tilde{z})$ be two points in $\Ge_n$ and $\zeta\in \mathfrak{so}(n)$ such that $g^{-1}\star \tg=(\tx-x,\zeta)$. There exists a successful coupling of Brownian motions $\left(\bbb_t,\tbbb_t\right)_t$ on $\Ge_n$ starting from $(g,\tg)$. Moreover, denoting $\tau:=\inf\{t\geq 0\,|\,\bbb_t=\tbbb_t\}$, for all $t\geq \beta_n\|x-\tx\|_2^2$:
    \begin{equation}\label{eq:chap1 couplingRateG}
    \pr(\tau>t)\leq \left(C_1(n)\frac{||x-\tx||_2}{\sqrt{t}}+C_2(n)\frac{||\zeta||_2}{t}\right)\end{equation}
    where: \begin{itemize}
        \item  $\beta_n=(n-1)^{\frac{3}{2}}b_n$ with $b_n{=2\sqrt{\pi n}}
$;
\item $C_1(n)
    =4\sqrt{\beta_n}(n-1)+\frac{1}{\sqrt{\pi}}$;
    \item $C_2(n)=2\beta_n$;
    \item $\|z\|_2=\sqrt{\sum\limits_{1\leq i<j\leq n}z_{i,j}^2}$.\end{itemize}
\end{thm}

{\begin{NB}
Denote $\bbb_t=(X_t,z_t)$ and $\tbbb_t=(\tX_t,\tz_t)$.
For all $1\leq i,j \leq n$ such that $i\neq j$, $(B_t^{i,j})_t:=(X_t^i,X_t^j,z_t^{i,j})_t$ (resp. $(\tilde{B}_t^{i,j})_t:=(\tX_t^i,\tX_t^j,\tz_t^{i,j})_t$) is a Brownian motion on the Heisenberg group. Using \emph{Theorem 3.6} from~\cite{banerjee2017coupling}, {for $t$ large enough}, there exists $C_1>0$ such that:
\begin{equation}
d_{TV}\left(\mathcal{L}(\bbb_t^{i,j}),\mathcal{L}(\tbbb_t^{i,j})\right)\geq C_1\left(\frac{\sqrt{(x_i-\tx_i)^2+(x_j-\tx_j)^2}}{\sqrt{t}}\mathbb{1}_{\{(x_i,x_j)\neq(\tx_i,\tx_j)\}}+\frac{|\zeta_{i,j}|}{t}\mathbb{1}_{\{\zeta_{i,j}\neq 0\}}\right).
\end{equation}
In particular, $C_1$ does not depend on the starting points $g$, $\tg$. Thus, {for $t$ large enough}, there exists a constant $C_2$ (not depending on $g$, $\tg$) such that:
\begin{align}
d_{TV}\left(\mathcal{L}(\bbb_t),\mathcal{L}(\tbbb_t)\right)&\geq \max\limits_{1\leq i<j\leq n} d_{TV}\left(\mathcal{L}(\bbb_t^{i,j}),\mathcal{L}(\tbbb_t^{i,j})\right)\notag\\
&\geq C_2\left(\frac{\|x-\tx\|_2}{\sqrt{t}}\mathbb{1}_{\{x\neq \tx\}}+\frac{\|\zeta\|_2}{t}\mathbb{1}_{\{x=\tx\}}\right).
\end{align}
The coupling from Theorem \ref{thm: successfulGn} is efficient on $\Ge_n$.
\end{NB}}

Compared to the strategy used for the Heisenberg group ($n=2$) in \cite{banerjee2017coupling}, the main difficulty for $n\geq 3$ is to deal with $\frac{n(n-1)}{2}$ Lévy areas instead of just one. From a quick look, there is no reason that moving one area $z_t^{i,j}$ will not impact another one. 
As in \cite{banerjee2017coupling}, the first step of this strategy is to apply a reflection coupling on the horizontal parts $(X_t,\tX_t)_t$ until the Brownian motions are in the same fiber. The second step consists in looking at the decomposition of the Brownian motions $(X_t)_t$ and $(\tX_t)_t$ using the Karhunen-Loève expansion of the Brownian bridge and coupling two by two the Gaussian coefficients appearing with this expansion. We then act on the skew-symmetric matrices $(z_t)_t$ and $(\tz_t)$ to couple the matrices one line after another. In particular we use properties of the inverse Wishart distribution.

A first corollary of Theorem \ref{thm: successfulGn} is an estimate of the total variation distance between the laws of the Brownian motions (Corollary \ref{Cor: dtvGnSucces}). We also obtain the existence of a successful coupling (see Corollary \ref{cor: couplingAllHomogeneous}) and an estimate of the Total Variation distance (see Corollary \ref{Cor: dtvGnSucces2}) for all homogeneous Carnot groups by using lifts. 
As for the Heisenberg group, using \eqref{eq: Chap1SemigpIn}, we also obtain results for the horizontal and vertical gradient of the heat semi group (see Corollary \ref{cor: horGradGnSucces} and Corollary \ref{cor: vertGradGnSucces}).

The second main result (see Theorem \ref{thm: sortieDomaine}) is an estimate of $\pr\left(\tau>\tau_{D}(\bbb)\wedge\tau_{D}(\tbbb)\right)$,  $\tau_D(\bbb)$ (resp. $\tau_D(\tbbb)$) denoting the first exit time of a domain $D\subset\Ge_n$ for the process $(\bbb_t)_t$ (resp. $(\tbbb_t)_t$). Although the strategy is not so different from the one developed in \cite{banerjee2017coupling}, some parts of the proof need extra care.
On $\Ge_n$ we then directly obtain estimates similar than (\ref{eq: Chap1 GradientDomainHeisenberg}) for the horizontal gradient of harmonic functions (Corollary \ref{cor: osculateur1} and Corollary \ref{cor: osculateur2}) as well as the Cheng Yau inequality (Corollary \ref{cor: cheng-Yau}). We also show how to extend these results to all homogeneous step $2$ Carnot groups. 

Note that the Cheng Yau inequality has been proven for all Carnot groups by Baudoin, Gordina and Mariano in~\cite{baudoin2019cheng} by using analytic methods.

\begin{NB}
{In a work in collaboration with Marc Arnaudon, Michel Bonnefont and Delphine Féral~\cite{OneSweep}, we propose two other strategies for a non co-adapted coupling of Brownian motion on $\Ge_n$.
\begin{itemize}
    \item The first one define a coupling of Brownian motion but only on a finite (and deterministic) interval of time $[0,T]$. As for the coupling of Theorem \ref{thm: successfulGn}, the strategy is using an expansion of the Brownian motion. There are however three major differences. The first one is that the couplings of the Gaussian coefficients appearing in the expansion are chosen such that only $\pr(\bbb_T\neq\tbbb_T)$ has a good order, i.e., it satisfies
    \begin{equation}\label{eq: OnesweepInequality}
        \pr(\bbb_T\neq\tbbb_T)\leq C_1(n)\frac{\|x-\tx\|_2}{\sqrt{t}}+C_2(n)\frac{\|\zeta\|_2}{T}.
    \end{equation} The second one is that the expansion of the Brownian motion in~\cite{OneSweep} uses Legendre polynomials. The third one is that the horizontal and vertical parts of the processes are considered together in a unique step. As $d_{TV}(\bbb_T,\tbbb_T)\leq\pr(\bbb_T\neq\bbb_T')$, this coupling provides the same inequalities than in Corollaries \ref{Cor: dtvGnSucces}, \ref{Cor: dtvGnSucces2}, \ref{cor: horGradGnSucces} and \ref{cor: vertGradGnSucces}. Moreover, in \cite{OneSweep}, the constants $C_1(n)$ and $C_2(n)$ are improved. \
    
    It is however important to notice that the coupling from \cite{OneSweep} is not successful. 
    As we cannot compare the first coupling time to the first exit time from a domain, this coupling from \cite{OneSweep} does not provide results comparable to Corollaries \ref{cor: osculateur1}, \ref{cor: osculateur2} and \ref{cor: cheng-Yau}. It seems that a simple repetition of the strategy with geometrically growing steps, using reflection coupling and concatenation of Brownian motions as in Proposition \ref{Brique1}, should also lead to a successful coupling at least when the two processes are starting from the same fiber. Nevertheless it may be more difficult to compare the first coupling time with the first exit time from a domain.
    \item The second strategy developed in \cite{OneSweep} is not exactly a coupling of Brownian motions. The idea is to couple two random variables $\bbb_T$ and $\tbbb_T$ such that $\bbb_T$ has the distribution of a Brownian motion at time $T$ (starting from $g$) whereas, only under an absolutely continuous change of probability, does $\tbbb_T$ have the distribution of a Brownian motion at time $T$ (starting at $\tg$). This allows us to construct the two variables such that $\tbbb_T$ is always equal to $\bbb_T$. Girsanov Theorem permits then to measure the distance between the corresponding semi groups. The applications as well as the method are then different than the one developed in the present paper.
\end{itemize}}

\end{NB}
\subsection{Organisation of the paper}
In Section \ref{Section: free Carnot}, we give some preliminaries about the SubRiemannian structure of the free, step $2$ Carnot groups $\Ge_n$ and its associated subRiemannian Brownian motions. In Section \ref{Section:Brownian BridgeGn} we describe the coupling strategy and give the proof of Theorem \ref{thm: successfulGn}. We also state the Total Variation distance estimate on $\Ge_n$ (Corollary \ref{Cor: dtvGnSucces}). The generalisation to all homogeneous step $2$ Carnot groups and the estimates for the Total Variation distance are dealt with in Section \ref{section: homogeneous}. In Section \ref{sec: coupling Vs Exit}, we compare the first coupling time and the first exit times from a domain (Theorem \ref{thm: sortieDomaine}) for this coupling strategy on free, step $2$ Carnot groups. Finally, we present the gradient inequalities obtained with this coupling in Section \ref{Sec: gradientEstimates}.
\section{The free, step $2$ Carnot groups}\label{Section: free Carnot}
\subsection{SubRiemannian structure on the free, step $2$ Carnot groups}\label{subsec: Chap1 Carnot}

For all $a\in\Ge_n$, define the left-translation $\trans_a:=g\in \Ge\mapsto a\star g\in \Ge$. A vector field $\bX$ on $\Ge_n$ is called left-invariant if, for all $g, \ a\in \Ge$ and for every smooth function $f$,
$\bX\left(f\circ \trans_a\right)(g)=\bX\left(f\right)({\trans_a(g)})$.

Consider $(\bX_1,\hdots\bX_n)$ the "horizontal left-invariant vector fields" such that $\bX_i(0)=\partial_{x_i}$ for all $1\leq i\leq n$. The horizontal subbundle $\mathcal{H}:=\Span\{\bar{X}_1,\hdots,\bar{X}_n\}$ satisfies the Hörmander condition, i.e., $(\bX_i)_{1\leq i\leq n}$ and $\left([\bX_i,\bX_j]\right)_{1\leq i<j\leq n}$ generate the tangent bundle $T\Ge_n$ (where $[\cdot,\cdot]$ denotes the usual Lie bracket operation). Using the Chow Theorem, we can then endow $\Ge_n$ with a structure of connected subRiemannian manifold.

We recall that the subRiemannian metric, also called, Carnot Carathéodory distance is defined for all $g,\tg\in \Ge_n$ by 
\begin{equation*}
		d_{cc}(g,\tg):=\inf\limits_{\gamma}L(\gamma)\text{ with }L(\gamma)=\int_I \sqrt{\langle\dot{\gamma}(t),\dot{\gamma}(t)\rangle_{\mathcal{H}_{\gamma(t)}}}dt
	\end{equation*}
	where $\gamma$ ranges over the horizontal curves joining $g$ and $\tg$, i.e.,  the smooth curves $\gamma:I\subset \mathbb{R}\to G$ such that $\dot{\gamma}(t)\in \mathcal{H}_{\gamma(t)}$. Here, for all $g\in \Ge_n$, the inner-product $\langle\cdot,\cdot\rangle_{\mathcal{H}_{g}}$ is defined such that $\left(\bX_1(g),\hdots,\bX_n(g)\right)$ is an orthonormal basis of $\mathcal{H}(g)$.

By construction of the horizontal subbundle, the Carnot Carathéodory distance $d_{cc}$ is left-invariant.
{\begin{NB}\label{NB: distanceHorizontale}
By decomposing any horizontal curve $t\mapsto\gamma(t)$ in $\Ge_n$ into the form $(x(t),z(t))\in\mathbb{R}^n\times\mathfrak{so}(n)$, and by considering explicit representations of the horizontal left-invariant vector fields, it is not difficult to prove that for any $g=(x,z)$, $\tg=(\tx,\tz)\in \Ge_n$:
 \begin{equation}\label{distanceHorizontale}
     \|x-\tx\|_2\leq d_{cc}(g,\tg).
 \end{equation}
\end{NB}}

\subsection{Definition of an equivalent pseudo-distance}\label{subsec: pseudo-distance}

As it can be difficult to work with the Carnot-Carathéodory distance, we introduce an equivalent pseudo-distance. (This pseudo-distance satisfies the properties of positivity and symmetry. The triangular inequality is replaced by a pseudo-triangular inequality.).

For $z=\left(z_{i,j}\right)_{1\leq i<j\leq n}\in\mathfrak{so}(n)$, denote: \begin{equation*}||z||_p:=\left(\sum\limits_{1\leq i<j\leq n}|z_{i,j}|^p\right)^{\frac{1}{p}}.\end{equation*} In fact, for all $z\in\mathfrak{so}(n)$, $\|z\|_2$ is, up to the multiplicative constant ${\sqrt{2}}$, the usual Hilbert-Schmidt norm of $z$. Note that, for any $x,\tx\in \mathbb{R}^n$, we have $\|x\symp \tx\|_2\leq \|x\|_2\|\tx\|_2$.

Define the pseudo-norm $\|\cdot\|_{\Ge_n}$:  
  \begin{equation}\label{eq: Chap1HomogeneousNorm}
   	||(x,z)||_{\Ge_n}:=\sqrt{||x||_2^2+\|z\|_2} \text{ for all } (x,z)\in \mathbb{R}^n\times\mathfrak{so}(n).
\end{equation}

Here we make the choice of quadratic norms. This way, $\|\cdot\|_{\Ge_n}$ is invariant under any change of basis on $\mathbb{R}^n$ in the sense that for any $P\in\mathcal{O}(n)$, $\left\|\left(Px,PzP^{t}\right)\right\|_{\Ge_n}=\left\|\left(x,z\right)\right\|_{\Ge_n}$.

From this pseudo-norm we can define the pseudo-distance 
\begin{equation}
\Big{\{}\begin{array}{ccc}
		\bd:\Ge_n\times \Ge_n&\to &\mathbb{R}\\
		(g,\tg)&\mapsto &\|g^{-1}\tg\|_{\Ge_n}
	\end{array}.
	\end{equation}
In particular, it satisfies the following pseudo-triangular inequality for all $(x,z)$, $(\tx,\tz)$, $(y,v)\in \Ge_n$:
	\begin{equation}\label{eq: pseudo-triangle}
	    	\bd\left((x,z),(\tx,\tz)\right)\leq {\frac{3}{2}}\big(\bd\left((x,z),(y,v)\right)+\bd\left((y,v),(\tx,\tz)\right)\big)
	\end{equation}



As $\|\cdot\|_{\Ge_n}$ and $g\in \Ge_n\mapsto d_{cc}(0,g)\in\mathbb{R}^+$ define two homogeneous norms on $\Ge_n$, the pseudo-distance $\bd$ and the subRiemannian distance $d_{cc}$ are equivalent (see \emph{Proposition 5.1.4} in \cite{bonfiglioli2007stratified}):
there exist {$m_1(n)$} and {$m_2(n)$} two positive constant such that for all $g,\tg\in \Ge_n$:
\begin{equation}\label{eq: equivSeminorme}
{m_1(n)}\bd(g, \tg)\leq d_{cc}(g,\tg)\leq {m_2(n)}\bd(g, \tg).
\end{equation}

{\begin{NB}
As far as the author knows, estimates for $m_1(n)$ and $m_2(n)$ do not seem to be commonly computed. However, they will be used several times in the corollaries of Sections \ref{Sec: couplingAllHomogeneous} and \ref{Sec: gradientEstimates}, in particular when generalizing our results to homogeneous Carnot groups. In what follows we try to gather some results that can help to compute these estimates.
\begin{itemize}

\item For some particular values of $g$ and $\tilde{g}$, the subRiemannian distance can be explicitly computed. For $x\in\mathbb{R}^n$, it is well known that:
\begin{equation}\label{eq: disthor}
 d_{cc}(0,(x,0))=\|x\|_2.
 \end{equation}
 For $z\in\mathfrak{so}(n)$, denote by $\lambda_1>\hdots>\lambda_r$ the absolute values of the non zero eigenvalues of $z$. \emph{Proposition 13.19} from~\cite{BoscainAgrachev} states:
\begin{equation}\label{eq: distvert}
    d_{cc}(0,(0,z))^2=4\pi \sum\limits_{i=1}^r i\lambda_i.
\end{equation}
From \eqref{eq: distvert}, using the fact that $\Tr(z^tz)=2\|z\|_2^2$ we get the equivalence relation:
\begin{equation}\label{eq: distvert2}
	  4\sqrt{2}\pi\|z\|_2\leq d_{cc}(0,(0,z))^2\leq {\frac{4\pi}{\sqrt{3}} \sqrt{n(n+1)(2n+1)}\|z\|_2}.
	\end{equation} 

\item  Using \eqref{eq: disthor} and \eqref{eq: distvert2}, we can obtain an estimate for the constant $m_2(n)$ intervening in \eqref{eq: equivSeminorme}:
\begin{equation}\label{equib bd dcc}
	  d_{cc}(g,\tg)\leq d_{cc}(0,(\tx-x,0))+d_{cc}(0,(0,\zeta))
	  \leq 2\sqrt{2\pi}\left(\frac{n(n+1)(2n+1)}{3}\right)^{\frac{1}{4}}
	  \bd(g,\tg)
	\end{equation}
	where $\zeta\in\mathfrak{so}(n)$ such that $g^{-1}\tg=(\tx-x,\zeta)$.
	Thus: 
	\begin{equation}
	m_2(n)\leq  2\sqrt{2\pi}\left(\frac{n(n+1)(2n+1)}{3}\right)^{\frac{1}{4}}.
	\end{equation}
	
\item It seems to be more difficult to obtain estimates for the constant $m_1(n)$. By considering projections of $\Ge_n$ on the Heisenberg group, we can however state that
\begin{equation}
    m_1(2)\sqrt{\frac{2}{n(n-1)}}\leq m_1(n).
\end{equation}
Indeed, let $\gamma:t\in[0,1]\mapsto (x(t),z(t))\in\Ge_n$ be an horizontal curve joining $g$ and $\tilde{g}$. For any $1\leq i<j\leq n$, $\gamma^{i,j}(t):=\begin{pmatrix}
x_i(t)\\
x_j(t)\\
z_{i,j}(t)
\end{pmatrix}$ is an horizontal curve on the Heisenberg group $\hh$ and $L(\gamma^{i,j})\leq L(\gamma)$. Thus, $d_{cc}\left((x_i,x_j,z_{i,j}),(\tx_i,\tx_j,\tz_{i,j})\right)\leq d_{cc}(g,\tg)$. Using the equivalence relation \eqref{eq: equivSeminorme} on the Heisenberg group ($n=2$), we get:
\begin{equation}
    m_1(2)\sqrt{(n-1)\|\tx-x\|_2^2+\|\zeta\|_2}\leq \sqrt{ \frac{n(n-1)}{2}}d_{cc}(g,\tg).
\end{equation}

\end{itemize}
\end{NB}}

 \subsection{Subelliptic Brownian motion}
We now introduce the sub-Laplacian operator: \begin{equation}\label{eq: Chap1 GenInf}
 	L:=\frac{1}{2}\sum_{1}^{n}\bar{X_i}^2.
	\end{equation}
	Because the Hörmander condition is satisfied, this operator is hypoelliptic. It is also a diffusion operator.
	\begin{defi}
	The {subelliptic Brownian motion} on $\Ge_n$ is defined as the diffusion process whose infinitesimal generator is $L$. 
	It satisfies the stochastic differential equation:
\begin{equation}\label{StratoMBCarnot}
    d\bbb_t=\sum\limits_{i=1}^{n}\bX_i(\bbb_t)\circ dB_t^i
\end{equation} where $\circ d$ denotes the Stratonovitch differential, $\left(B_t^1,\hdots,B_t^n\right)$ a Brownian motion on $\mathbb{R}^n$ and $(\bX_i)_{1\leq i\leq n}$ still denote the horizontal left-invariant vector fields on $\Ge_n$.

 In particular, the subelliptic Brownian motion starting from $g=(x,z)\in \Ge_n$ can be written $(\bbb_t)_t:=(X_t,z_t)_t$ with $(X_t)_t$ a Brownian motion starting at $x$ on $\mathbb{R}^n$ and $z_t:=z+\frac{1}{2}\int_0^t X_s\symp dX_s$. In fact, 
 \begin{equation*}
 \mathcal{A}_t^{i,j}:=z_t^{i,j}-z_0^{i,j}=\int_0^t X_s^idX_s^j-\int_0^t X_s^j dX_s^j
 \end{equation*}
 is the Lévy area between $(X_s^i)_{s\leq t}$ and $(X_s^j)_{s\leq t}$.
 \end{defi}
 \begin{NB}\label{RemChap1 Heisenberg}
	For $n=2$, $\Ge_2$ can be identified with the well known Heisenberg group $\hh=(\mathbb{R}^3,\star)$ where $(x_1,x_2,z)\star(\tx_1,\tx_2,\tz)=(x_1+\tx_1,x_2+\tx_2,z+\tz+\frac{1}{2}(x_1\tx_2-\tx_1x_2))$.
	
	On $\hh$, the Brownian motion can be written $(\bbb_t)_t=(X_t,z_t)_t$ with $(X_t)_t$ a Brownian motion on $\mathbb{R}^2$ and $z_t-z_0$ the signed area swept by $(X_s)_{s\leq t}$ with respect to the origin of the coordinate system, i.e., the Lévy area $\int_0^t X_s^1dX_s^2-\int_0^t X_s^2 dX_s^1$. 
	
\end{NB}
 Denote by $(P_t)_t$ the {heat semi-group}, i.e., the semi-group associated to $L$. For all $g\in\Ge_n$, consider a Brownian motion $(\bbb_t)_t$ starting at $g$. Then for all bounded measurable function $f$ on $\Ge_n$:
	\begin{equation}
	P_tf(g)=\esp[f(\bbb_t)].
	\end{equation}

\section{The non co-adapted coupling strategy for $\Ge_n$}\label{Section:Brownian BridgeGn}
In this section we prove Theorem \ref{thm: successfulGn}. Set $n\geq 2$. In the remainder of the article we use the following notations. Let $(\bbb_t)_t=(X_t,z_t)_t$ and $(\tbbb_t)_t=(\tX_t,\tz_t)_t$ be two Brownian motions on $\Ge_n$ starting at $g=(x,z)$ and $\tg=(\tx,\tz)$.
We have $\bbb_t^{-1}\tbbb_t=(\tX_t-X_t,\zeta_t)$ with $\zeta_t:=\tz_t-z_t-\frac{1}{2}X_t\symp \tX_t\in\mathfrak{so}(n)$. In particular, for any $1\leq i\neq j\leq n$, using the Itô formula, we can write:
\begin{align}\label{eq: coordonneesZeta}
d\zeta_t^{i,j}&=\frac{1}{2}\tX_t^id\tX_t^j-\frac{1}{2}\tX_t^jd\tX_t^i-\left(\frac{1}{2}X_t^idX_t^j-\frac{1}{2}X_t^jdX_t^i\right)-\frac{1}{2}d\left(X_t^i\tX^j_t-X_t^j\tX_t^i\right)\notag\\
&=\tX_t^id\tX_t^j-X_t^idX_t^j+\frac{1}{2}d\left(\left(\tX^i_t+X_t^i\right)\left(X_t^j-\tX_t^j\right)\right).
\end{align}
Note that if $X_T=\tX_T$, $\zeta_T=\tz_T-z_T$.

\subsection{The Brownian motions start from the same fiber}\label{subsec: memefibreGn}
We first describe the strategy for the case when $x=\tx$, i.e., when the Brownian motions start from the same fiber. The main idea is to work on the skew-symmetric matrix $(\zeta_t)_t$ line by line. For each line we will apply the following Proposition:
\begin{prop}\label{Brique1}
	Let $g=(x,z)$, $\tilde{g}=(\tilde{x},\tilde{z})$ be two points in $\Ge_n$ such that $x=\tilde{x}$.
Let $1\leq i_0\leq n-1$. For two Brownian motions $\bbb_t$ and $\tbbb_t$, denote $v^{i_0}_t:=\left(\zeta^{i_0,i}_t\right)_{i\in\{1,\hdots,n\}\setminus\{ i_0\}}\in\mathbb{R}^{n-1}$ and \begin{equation*}\bar{\tau}_{i_0}:=\inf\{t\geq0 \ | \ X_t=\tX_t\text{ and } z_t^{i_0,i}=\tilde{z}_t^{i_0,i} \ \forall 1\leq i\leq n \}.\end{equation*}
 There exists a non co-adapted coupling $(\bbb_t,\tilde{\bbb}_t)_t$ such that $\bar{\tau}_{i_0}$ is a.s. finite and satisfies for all $t\geq \|v^{i_0}_0\|_2$:
\begin{equation}\label{eq: InegaliteBrique}
\pr\left(\bar{\tau}_{i_0}>t\right)\leq\frac{\|v^{i_0}_0\|_2}{t}\times b_n
\end{equation}
{with $b_n={2\sqrt{\pi n}}
$.}
Moreover, for $l,i\in\{1,\hdots n\}\setminus\{i_0\}$, we have $\zeta_t^{l,i}=\zeta_0^{l,i}$ for all $t\in[0,\bar{\tau}_{i_0}]$.
\end{prop}	
\begin{NB}
As $(\zeta_t)_t$ is skew-symmetric, at time $\bar{\tau}_{i_0}$ the $i_0^{th}$ line \textbf{and} the $i_0^{th}$ column are null.
\end{NB}
\begin{NB}\label{Rem: Brique}
As $b_n>1$
for all $n\geq 2$, Inequality (\ref{eq: InegaliteBrique}) can be replaced by 
\begin{equation*}
\pr\left(\bar{\tau}_{i_0}>t\right)\leq\left(\frac{\|v^{i_0}_0\|_2}{t}\times b_n\right)\wedge 1\text{ for all }t\geq 0.
\end{equation*}

\end{NB}
	
To prove Proposition \ref{Brique1}, we recall the well known estimate of the first hitting time of a standard Brownian motion:
\begin{lemme}\label{hittingTime}
	Set $a\in\mathbb{R}$. Let $(W_t)_t$ be a standard Brownian motion on $\mathbb{R}$. Denote by $D_a:=\inf\{t >0\ | \ W_t=a\}$ the first hitting time of $a$ by $(W_t)_t$. Then for all $t>0$:
	\begin{equation*}
	\pr(D_a>t)\leq \left(\sqrt{\frac{2}{\pi}}\frac{|a|}{\sqrt{ t}}\right)\wedge 1.
	\end{equation*}
\end{lemme}
	
\begin{proof}[Proof of Proposition \ref{Brique1}]
	For the proof we can suppose without lost of generality that $i_0=1$. Then $v_t^1=\begin{pmatrix}
\zeta_t^{1,2}\\
\vdots\\
\zeta_t^{1,n}
\end{pmatrix}$. Let $T>0$ and $\ve:=\frac{v_0^1}{\|v_0^1\|_2}$. 

We first choose $X_t^i=\tilde{X}^i_t$ for $2\leq i \leq n$. In particular $\zeta_t^{l,i}$ stays constant for $2\leq l<i \leq n$.

We now define the coupling $({X}_t^1,\tX_t^1)_t$. Let $(\xi_j)_{j\geq 1}$ (resp. $(\txi_j)_{j\geq 1}$) be a sequence of independent variables with the same distribution $\mathcal{N}(0,1)$. Using the Karhunen-Loève expansion we can define two Brownian bridges $(\bbr_t)_{0\leq t\leq T}$ and $(\tbbr_t)_{0\leq t\leq T}$ on the interval of time $[0,T]$:
    \begin{equation}\label{eq: Chap1Karhunen Loeve}
    \bbr_t=\sqrt{T}\sum\limits_{j\geq1}\xi_j\frac{\sqrt{2}}{j\pi}\sin\left(\frac{j\pi t}{T}\right)
    \text{ and } {\tbbr}_t=\sqrt{T}\sum\limits_{j\geq1}\tilde{\xi}_j\frac{\sqrt{2}}{j\pi}\sin\left(\frac{j\pi t}{T}\right).
    \end{equation}
    
    Then, for $t\in[0,T]$
    \begin{equation}\label{eq: Chap1Bridges}
    X_t^1:=x_1+\bbr_t+\frac{t}{\sqrt{T}}\xi_0 \text{ and } \tX_t^1:=\tx_1+{\tbbr}_t+\frac{t}{\sqrt{T}}\tilde{\xi}_0
    \end{equation}
    define two Brownian motions starting from $x_1$ and $\tx_1$ respectively.
    In particular, we have for $t\in[0,T]$:
    \begin{equation*}
		\tX_t^1-{X}_t^1=\sum\limits_{j\geq 1}\frac{\tilde{\xi}_j-\xi_j}{2}\times 2\frac{\sqrt{2T}}{j\pi}\sin\left(\frac{j\pi t}{T}\right)+\left(\txi_0-\xi_0\right)\frac{t}{\sqrt{T}}.
		\end{equation*}

	Let {$m\geq n+1$}. We choose $\xi_j=\tilde{\xi}_j$ for $j=0$ and for {$j\geq m$}. With these choices, for all $2\leq i\leq n$, using (\ref{eq: coordonneesZeta}) we get:
	\begin{equation*}
	\zeta_t^{1,i}=\zeta_0^{1,i}+\int_0^t\left(\tX_s^1-X_s^1\right)dX_s^i=\zeta_0^{1,i}+\sum\limits_{j=1}^{{m}}\frac{\tilde{\xi}_j-\xi_j}{2}K_{(i-1),j}(t)
	\end{equation*}
	with: $K_{(i-1),j}(t):=2\frac{\sqrt{2T}}{j\pi}\int_0^t\sin\left(\frac{j\pi s}{T}\right)dX_s^i$. 
		Note that, as we consider the coordinates $\zeta_t^{1,i}$ for ${2\leq i\leq n}$, we introduce a shift of the index $i$ to define $K_{(i-1),j}(t)$. This will simplify the notation in the remainder of the proof.
		For $j,l>0$, we have: \begin{align*}
	\esp\left[\int_0^T\sin\left(\frac{j\pi s}{T}\right)dX_s^i\int_0^T\sin\left(\frac{l\pi s}{T}\right)dX_s^i\right]&=\int_0^T\sin\left(\frac{j\pi s}{T}\right)\sin\left(\frac{l\pi s}{T}\right)ds\\
	&=\frac{1}{2}\int_0^T\left(\cos\left(\frac{(j-l)\pi s}{T}\right)-\cos\left(\frac{(j+l)\pi s}{T}\right)\right)ds\\
	&=\begin{cases}
	0 \ &\text{if}\ j\neq l\\
	\frac{T}{2} \ &\text{if}\ j=l.
	\end{cases}
	\end{align*}
 Thus, we have $K_{(i),j}(t)=\frac{2T}{j\pi}G_{(i),j}(t)$ for all {$1\leq i\leq n-1$, $1\leq j\leq m$} with $(G_{(i),j}(T))_{i,j}$ some independent standard Gaussian variables.

	{Denote $\Sigma:=\diag\left(1,\frac{1}{4},\hdots,\frac{1}{m^2}\right)$ the $m\times m$ diagonal matrix with $\left(1,\frac{1}{4},\hdots,\frac{1}{m^2}\right)$ on the diagonal and $R(T)=\left(G_{(i),j}(T)\right)_{\substack{ 1\leq i\leq n-1\\ 1\leq j\leq m}}$. In particular, denoting $\xi:=\left(\xi_j\right)_{1\leq j\leq m}\in\mathbb{R}^{m}$ and $\txi:=(\txi_j)_{1\leq j\leq m}\in\mathbb{R}^{m}$, we have:
	\begin{equation*}
	v^{1}_t=v^{1}_0+\frac{2T}{\pi}R(t)\Sigma^{\frac{1}{2}}\frac{\tilde{\xi}-\xi}{2}.
	\end{equation*}
	Note that $R(T)$ is measurable with respect to the $\sigma$-field $\sigma\left(X_t^j\, ,0\leq t\leq T\,, 2\leq j\leq n\right)$ and thus will be independent of $(X_t^1)_t$ and $(\tX_t^1)_t$.}

	{ The distribution of $R(T)R(T)^t$ is known: it is a $\mathcal{W}_m(n-1,I_{m})$  standard Wishart distribution with $m\geq n+1$ degrees of freedom\footnote{{Let A be a $n\times n$ random matrix. Using the notations of~\cite{Muirhead}, we say that $A$ has the Wishart distribution with $m$ degrees of freedom and with the covariance matrix $\Xi$ if it can be written under the form $A=XX^t$ with $X$ a $n\times m$ matrix whose lines are i.i.d. with distribution $\mathcal{N}(0,\Xi)$. This distribution is written $\mathcal{W}_m(n,\Xi)$.}} (see for example \emph{Chapter 3} from~\cite{Muirhead} for some properties of the Wishart distribution). A first consequence is that $R(T)R(T)^t$ has a density and is thus a.s. invertible.} Denote then $w:={R(T)}^t\left({R(T)}{R(T)}^t\right)^{-1}v\in\mathbb{R}^m$. In particular, $R(T)w=v$.
{	Define {$f_1:=\frac{\Sigma^{-\frac{1}{2}}w}{||\Sigma^{-\frac{1}{2}}w||_2}$. Suppose that $\frac{\tilde{\xi}-\xi}{2}=\rho f_1$. Then, \begin{equation*}R(T)\Sigma^{\frac{1}{2}}\frac{\tilde{\xi}-\xi}{2}=\frac{\rho}{||\Sigma^{-\frac{1}{2}}w||_2\|v_0^1\|_2}v_0^1
\end{equation*}
and the coupling is successful at time $T$ if and only if:
	\begin{equation}\label{eq: xi-xi/2}
 	\rho=-\frac{\pi}{2T}\|v_0^1\|_2\|\Sigma^{-\frac{1}{2}}w\|_2.
 	\end{equation}}}
	Let $W_t$ be a ${m}$-dimensional Brownian motion starting from $0$ and 
{$\sigma:=\inf\{t\geq0 \ | \ W_t^1=\frac{\pi}{2T}\|v_0^1\|_2||\Sigma^{-\frac{1}{2}}w||_2\}$.}
	We define $\tilde{W}_t$ such that $W^j_t=\tilde{W}^j_t$ for all $j\geq 2$
	and \begin{equation*}
	\tilde{W}_t^1:=
	\begin{cases}
	-W_t^1 &\text{ if } t\leq \sigma \\\
	W_t^1 -2W_{\sigma}^1  &\text{else}
	\end{cases}.
	\end{equation*}
	In particular, $\frac{\tilde{W}_t^1-W_t^1}{2}=-W^1_{\sigma\wedge t}$. 
	
	We can now define the coupling for the Gaussians vectors $(\xi,\tilde{\xi})$. We choose:
	$
	\xi=\sum\limits_{l=1}^{{m}}W_1^l f_l$ and $
\tilde{\xi}=\sum\limits_{l=1}^{{m}}\tilde{W}_1^l f_l$. As $(W_t)_t$ is independent of $(f_l)_{1\leq l\leq {m}}$, this provides a well defined coupling for $\left(\bbb_t,\tbbb_t\right)_t$.
	In particular, 
	{$\frac{\tilde{\xi}-\xi}{2}=-W^1_{\sigma\wedge 1}f_1=-W^1_{\sigma\wedge 1}\frac{\Sigma^{-\frac{1}{2}}w}{||\Sigma^{-\frac{1}{2}}w||_2}$. }
		According to (\ref{eq: xi-xi/2}) and by definition of $\sigma$, the coupling is successful at time $T$ if and only if $\sigma<1$.\\
	Note that:
	{	\begin{equation*}
	v_t^1=
	\|v_0^1\|_2\ve-\frac{2T}{\pi}\frac{W^1_{\sigma\wedge1}}{||\Sigma^{-\frac{1}{2}}w||_2} R(t)w
	\end{equation*}}
	and at time $t=T$: 
	\begin{equation}\label{v_T vs M_t}
	v_T^1=\left(\|v_0^1\|_2-\frac{2T}{\pi}\frac{W^1_{\sigma\wedge1}}{||\Sigma^{-\frac{1}{2}}w||_2}\right) \ve.
	\end{equation}
In particular, $\|v_T^1\|_2=\|v_0^1\|_2-\frac{W^1_{\sigma\wedge1}}{||\Sigma^{-\frac{1}{2}}w||_2}$ by definition of $\sigma$.

{We now take a look at the behaviour of the horizontal parts induced by this coupling. Denote by $(e_1,\hdots,e_{m})$ the canonical basis in $\mathbb{R}^{m}$. For $0< t< T$:
\begin{align*}
\tX_t^1-X_t^1&=\sum\limits_{j=1}^m -2\frac{\sqrt{2T}}{j\pi}\sin\left(\frac{j\pi t}{T}\right)W^1_{\sigma\wedge 1}\frac{\langle\Sigma^{-\frac{1}{2}}w,e_j\rangle}{||\Sigma^{-\frac{1}{2}}w||_2}\\
&=-2\frac{\sqrt{2T}}{\pi||\Sigma^{-\frac{1}{2}}w||_2}W^1_{\sigma\wedge 1}\langle w,u(t)\rangle\\
&=-2\frac{\sqrt{2T}}{\pi||\Sigma^{-\frac{1}{2}}w||_2}W^1_{\sigma\wedge 1}\left\langle R(T)^t\left(R(T)R(T)^t\right)^{-1}v,u(t)\right\rangle
\end{align*}
where $u(t)=\sum\limits_{j=1}^m \sin\left(\frac{j\pi t}{T}\right)e_j$ is deterministic and nonzero.
In particular, a.s., $X_t=\tX_t$ if and only if $R(T)^t\left(R(T)R(T)^t\right)^{-1}$ lies in the hyperplane $L_t^{-1}(0)$ where $L_t:=A\in\mathcal{M}_{m,n-1}\mapsto\langle Av,u(t)\rangle$. Since $\left(R(T)^t\left(R(T)R(T)^t\right)^{-1}\right)^t$ has a pseudo-inverse matrix-variate normal distribution, it admits a density (see~\cite{Pseudo-inverse-Z} \emph{Section 4} or~\cite{Moore-Penrose-HS} for more general results about pseudo-inverse random matrices). Then $\pr(X^1_t=\tX^1_t)=0$ for $0<t<T$. In particular the coupling can only occurs at time $T$. }

		For the rest of the proof, we will need to study an upper bound of $||\Sigma^{-\frac{1}{2}}w||_2^2$. {We have 
	\begin{equation*}
{2}\|w\|_2^2=\Tr(w^tw)=\Tr\left(v^t\left({R(T)}{R(T)}^t\right)^{-1}v\right).
\end{equation*}
Then, with similar computations than in \emph{Proposition 3.1} from \cite{OneSweep}, using the fact that the Wishart is standard, we have:
	\begin{equation*}
{2}\esp[\|w\|_2^2]\leq\frac{\Tr(v v^t)}{n-1}\esp\left[\Tr\left(\left({R(T)}{R(T)}^t\right)^{-1}\right)\right]
\end{equation*}
As in~\cite{OneSweep}, using the result from \emph{Example 3.1} by~\cite{PH21} about the first moment of the inverse of a standard Wishart, we get:
\begin{equation}\label{eq: espWishart}
{2}\esp[\|w\|_2^2]\leq \frac{\|v\|^2_2}{n-1}\frac{n-1}{{m-(n-1)-1}}=\frac{1}{{m-n}}.
\end{equation}
Then:
	\begin{equation}\label{eq: Wishart-chi2}
	    \esp\left[||\Sigma^{-\frac{1}{2}}w||_2\right]=\esp\left[\sqrt{w^{t}\Sigma^{-1}w}\right]\leq m\sqrt{\esp[\|w\|_2^2]}\leq \frac{m}{\sqrt{{2(m-n)}}}.
	\end{equation}}

	We now give the global construction of our successful coupling. As for the Heisenberg group, we define a partition of the time line with geometrically growing lengths: $t_0=0<t_1<...<t_k<...$ such that for all $k\geq 0$, $T_k:=t_{k+1}-t_k=\frac{\|v_0^1\|_2}{3}2^{k}$. We reproduce the above construction on each interval $[t_k,t_{k+1}]$ and we use the exponent "$(k)$" to distinguish the objects define above. Note that we do not update the vector $\ve$ at each iteration but keep it equal to $\frac{v_0^1}{\|v_0^1\|_2}$ all the time. 
	The obtained matrices {$R^{(k)}(T_k)$}
	are independent and the same is true for the Brownian motions $\left(W_t^{1,(k)}\right)_t$. Then we get $\tau>t_{N}$ if and only if $v_t^1\neq 0$ for all $t=t_k$, $k\leq N$ which is equivalent to $\sigma^{(k)}>1$ for all $k\leq N-1$. We define the process $(M_t)_t$ at least on $[0,t_{N}]$ such that:
{	\begin{equation*}
	    \left\{\begin{array}{l}
		M_0=\|v_0^1\|_2\\
		M_t=M_{t_k}-W_{\frac{t-t_k}{T_k}}^{1,(k)}\times\frac{2T_k}{\pi}\frac{1}{\left\|\Sigma^{-\frac{1}{2}}w^{(k)}\right\|_2}\text{ for }t_k\leq t\leq t_{k+1}.
	\end{array}\right.
	\end{equation*}}
	In particular, from \eqref{v_T vs M_t}, for all $k\leq M$, we have $v_{t_k}=M_{t_k}\ve$.
	
	As $(M_t)_t$ is a martingale, it can be written $M_t=\|v_0^1\|_2+\Bet_{S(t)}$ with $\Bet$ a Brownian motion starting from $0$ and {$S(t)=\sum\limits_{k=0}^{N-1} \left(\frac{2T_k}{\pi}\right)^2\frac{1}{\left\|\Sigma^{-\frac{1}{2}}w^{(k)}\right\|_2^2}\left(\frac{t-t_k}{T_k}\wedge 1\right)\mathbb{1}_{\{t\geq t_k\}}$.}
	By construction of $(M_t)_t$, we have $\bar{\tau}_1>t_{N}$ if and only if $D_{-\|v_0^1\|_2}>S(t_{N})$ with $D_{-\|v_0^1\|_2}:=\inf\{s>0\ | \ \Bet_s=-\|v_0^1\|_2\}$. Thus, using an estimate of the first hitting time of a Brownian motion (Lemma \ref{hittingTime}), we get $\pr(\bar{\tau}_1>t_{N})\leq \frac{\sqrt{2}\|v_0^1\|_2}{\sqrt{\pi}}\esp\left[\frac{1}{\sqrt{S(t_{N})}}\right]$.
	For $N\geq {1}$, using (\ref{eq: Wishart-chi2}), we have:

	{	\begin{align*}
	\esp\left[\frac{1}{\sqrt{S(t_{N})}}\right]
	&\leq\frac{\pi }{2T_N}\esp\left[\left\|\Sigma^{-\frac{1}{2}}w^{(N)}\right\|_2\right]\leq \frac{\pi}{2T_N}\frac{m}{\sqrt{{2(m-n)}}}.
	\end{align*}}

We get: 
\begin{equation*}
\pr(\bar{\tau}_1>t_N)\leq \frac{\sqrt{2\pi}\|v_0^1\|_2}{T_{N+1}}\frac{m}{\sqrt{{2(m-n)}}}.
\end{equation*}
	Thus, for $t\in[t_N,t_{N+1}[=\left[\|v_0^1\|_2\frac{2^{N}-1}{3},\|v_0^1\|_2\frac{2^{N+1}-1}{3}\right[$ with $N\geq 1$ we have:
	\begin{equation}\label{eq: couplingRateGeneralm}
	\pr(\bar{\tau}_1>t)=\pr(\bar{\tau}_1>t_N)\leq \frac{\sqrt{2\pi}\|v_0^1\|_2}{t}\frac{m}{\sqrt{{2(m-n)}}}.
	\end{equation}
	By construction $t_1=\frac{\|v_0^1\|_2}{3}$, and \eqref{eq: couplingRateGeneralm} is true for all $t\geq \frac{\|v_0^1\|_2}{3}$. In particular, \eqref{eq: couplingRateGeneralm} is optimal for $m=2n$ which ensues the expected result.
\end{proof}
\begin{NB}\label{NB: ComparaisonHeisenberg}
In \cite{banerjee2017coupling} where the case $n=2$ is treated; the choice is made to simply take $m=n-1$. The above strategy is also successful for all $n\geq 2$ with $m=n-1$. In this case, $R(T)$ is invertible and $w=R(T)^{-1}v$. In particular $\frac{1}{\|w\|_2^2}$ is a chi-square with 1 degree of freedom (see for example~\cite{Muirhead}, Chapter 3). As in \cite{banerjee2017coupling}, at least two iterations of the Brownian bridges coupling are then necessary to have $\frac{1}{\sqrt{S(t_N)}}$ integrable: for $N\geq 2$,
	\begin{align*}
	    S(t_{N})\geq \left(\frac{2T_{N-1}}{\pi(n-1)}\right)^2\left(\frac{1}{||w^{(N)}||^2_2}+\frac{1}{||w^{(N-1)}||^2_2}\right)
	\end{align*}
and \begin{equation*}\esp\left[\left(\frac{1}{||w^{(N)}||^2_2}+\frac{1}{||w^{(N-1)}||^2_2}\right)^{-\frac{1}{2}}\right]=\sqrt{\frac{\pi}{2}}.\end{equation*}
Then, for $t\in[t_N,t_{N+1}[=\left[\|v_0^1\|_2\frac{2^{N}-1}{3},\|v_0^1\|_2\frac{2^{N+1}-1}{3}\right[$ with $N\geq 2$ we have:
	\begin{equation}\label{eq: couplingRateCarre}
	\pr(\bar{\tau}_1>t)=\pr(\bar{\tau}_1>t_N)<\|v_0^1\|_2\frac{2\pi(n-1)}{t_{N+1}}< \|v_0^1\|_2\frac{2\pi(n-1)}{t}.
	\end{equation}
	As $t_2=\|v_0^1\|_2$, the inequality is true for all $t\geq \|v_0^1\|_2$.
\end{NB}

We can now give the global construction for the case when the Brownian motions start from the same fiber, i.e. $x=\tx$:
\begin{thm}\label{Couplage par lignes}
	Set $n\geq 2$. Let $g=(x,z)$, $\tilde{g}=(\tilde{x},\tilde{z})$ be two points in $\Ge_n$ and $\zeta\in\mathfrak{so}(n)$ such that $g^{-1}\star \tg=(\tx-x,\zeta)$. Suppose that $x=\tilde{x}$. Then there exists a successful non co-adapted coupling $(\bbb_t,\tilde{\bbb}_t)$ starting from $(g,\tg)$ such that, for all $t\geq (n-1)\|\zeta\|_2$:
	\begin{equation}\label{eq: InegaliteMemeFibre}
		\pr(\tau>t)\leq ||\zeta||_2\frac{\beta_n}{t} \text{ with }{\beta_n=b_n(n-1)^{\frac{3}{2}}}.
	\end{equation}  
\end{thm}
\begin{NB}\label{Rem: tGd par rapport à Z}
As $\beta_n \geq 1$ for all $n\geq 2$, (\ref{eq: InegaliteMemeFibre}) is trivially still true if $t<(n-1)\|\zeta\|_2$ and can be replaced by:
\begin{equation*}
    \pr\left(\tau>t\right)\leq\left(||\zeta||_2\frac{\beta_n}{t}\right)\wedge 1\text{ for all }t>0.
\end{equation*}
\end{NB}

\begin{proof}[Proof of Theorem \ref{Couplage par lignes}]
Let $g$, $\tilde{g}\in \Ge_{n}$. 
Denote $\tau_0:=0$ and, for $1\leq i\leq n-1$:
\begin{equation*} 
\tau_{i}:=\inf\left\{t>\tau_{i-1}\ | \  X_t=\tX_t \text{ and }\zeta_t^{i,j}=\tilde{\zeta}_t^{i,j} \ \forall j\in\{1,\hdots n-1\}\setminus\{i\}\right\}.
\end{equation*}
For all $1\leq i\leq n-1$, we apply the coupling from Proposition \ref{Couplage par lignes} for $i_0=i$ on $[\tau_{i-1},\tau_{i}]$. As $\zeta_t$ is skew-symmetric and by construction of the coupling, we have $\sum\limits_{i=1}^{n-1}\tau_{i}=\tau$. {In particular $\|v_{\tau_{i-1}}^i\|_2^2=\sum\limits_{j=i+1}^n\left(\zeta_0^{i,j}\right)^2$.}

If $t\geq (n-1)\|\zeta_0\|_2$, then $\frac{t}{n-1}\geq\|v_{\tau_{i-1}}^i\|_2$ for all $1\leq i\leq n-1$. We obtain:
\begin{align*}
	\pr\left({\tau}>t\right)&\leq \sum\limits_{i=2}^{n-1}\pr\left(\tau_{i}-\tau_{i-1}>\frac{t}{n-1},\tau_{l}-\tau_{l-1}\leq\frac{t}{n-1}\ \forall 1\leq l\leq i-1\right)\\
	&+\pr\left(\tau_{1}>\frac{t}{n-1}\right)\\
	&{\leq \sum\limits_{i=1}^{n-1}\|v_{\tau_{i-1}}^i\|_2\frac{(n-1)b_n}{t}\leq\sqrt{n-1} \|\zeta_0\|_2 \frac{(n-1)b_n}{t}.}
\end{align*} 
\end{proof}

\begin{NB}
As $\zeta$ is skew-symmetric, there exists a matrix $P_v$ such that  $\left(P_v^{t}\zeta P_v\right)_{i,j}\neq0$ only if $i$ is odd and $j=i+1$. In particular, $\|P_v^{t}\zeta P_v\|_2=\|\zeta\|_2$. Then applying the coupling from Proposition \ref{Couplage par lignes} to this new matrix, we can reduce the number of {lines} to be treated. In particular, we obtain the previous result with $\lfloor\frac{n}{2}\rfloor^{\frac{3}{2}}b_n$ instead of $\beta_n$.
\end{NB}
\begin{NB}
{The estimation of the coupling rate obtained in \eqref{eq: InegaliteMemeFibre} is enough to conclude that this vertical coupling is efficient. The obtained constant $\beta_n$ is however not good enough to say that it is maximal. Indeed, the coupling rate obtained with the maximal coupling of Luo and Neel in~\cite{luo2024nonmarkovian} is equal to $\frac{\|\zeta\|_2}{t}$. Even for $n\geq 3$, the estimate of the total variation distance in \emph{Theorem 3.2} from~\cite{OneSweep} is better than the one obtained here. Of course some estimates such as \eqref{eq: Wishart-chi2} are maybe not good enough. In particular it forces us to chose $m=2n$ when it would seem more efficient to release the constraints $\xi_k=\txi_k$ as much as possible, i.e., to take $m$ as large as possible.}
\end{NB}
\subsection{General case}
We can now deal with the general case, i.e., for $x\neq \tilde{x}$. 
\begin{proof}[Proof of Theorem \ref{thm: successfulGn}]
To obtain a successful coupling on $\Ge_n$, we are going to use first the reflection coupling until the time $\tau_0:=\inf\{t\geq0\ | \  X_t=\tilde{X}_t\}$ and then the coupling from Theorem \ref{Couplage par lignes}. We remind how to define the reflection coupling. There exists $P_h\in\mathcal{O}(n)$ 
 such that $(P_hx)_i=(P_h\tx)_i$ for all $2\leq i\leq n$, and $(P_hx)_1-(P_h\tx)_1=||x-\tx||_2$. We first define a coupling $\left({\uX}_t,\utX_t\right)$ of Brownian motions on $\mathbb{R}^n$ starting from $\left(P_hx,P_h\tx\right)$: we choose $\uX^i_t={\utX}^i_t$ for $2\leq i\leq n$ and $d\uX_t^1=-d\utX_t^1$. We are then able to deduce a coupling of Brownian motion on $\Ge_n$ such that $(X_t,\tX_t):=\left(P_h^{t}\uX_t,P_h^{t}\utX_t\right)$. In particular, we have $\tau_0=\inf\{t\geq0\ | \  \uX_t=\tilde{\uX}_t\}$. Thus $\tau_0$ has the distribution of the first hitting time $D_a$ of a Brownian motion with $a=\frac{1}{2}||x-\tx||_2$ and $\pr(\tau_0>t) \leq \left(\frac{||x-\tx||_2}{\sqrt{2\pi t}}\right)\wedge 1$ (see Lemma \ref{hittingTime}). At time $\tau_0$, the Brownian motions are in the same fiber and we can apply the coupling from Theorem \ref{Couplage par lignes} on $[\tau_0,\tau]$.
For $t>0$, we get:
\begin{align}\label{eq: couplingRate}
\pr(\tau>t)&=\pr\left(\tau>t,\tau_0\leq\frac{t}{2}\right)+\pr\left(\tau>t,\tau_0>\frac{t}{2}\right)\leq\pr\left(\tau-\tau_0>\frac{t}{2},\tau_0\leq\frac{t}{2}\right)+\pr\left(\tau_0>\frac{t}{2}\right)\notag\\
&\leq\esp\left[\pr\left(\tau-\tau_0>\frac{t}{2}\ \bigg|\	\|\zeta_{\tau_0}\|_2\right)\mathbb{1}_{\{\tau_0\leq\frac{t}{2}\}}\right]+\frac{||x-\tx||_2}{\sqrt{\pi t}}\notag\\
&\leq\esp\left[\left(||\zeta_{\tau_0}||_2\frac{2\beta_n}{t}\right)\wedge 1\right]+\frac{||x-\tx||_2}{\sqrt{\pi t}}.
\end{align}
We now need a good estimate of $\esp\left[||\zeta_{\tau_0}||_2\right]$. This can be done by using the following Lemma which will be proven just after the end of the current proof: 
\begin{lemme}\label{lemme: aires multiples2}
For every $m\geq \frac{1}{2}\|x-\tx\|_2^2$, we have:
\begin{equation*}
\esp\left[\frac{\left\|\zeta_{\tau_0}\right\|_2}{{m}}\wedge 1\right]\leq \frac{\|\zeta\|_2}{m}+\frac{2\sqrt{2}}{\sqrt{m}}(n-1)||x-\tx||_2.
\end{equation*}
\end{lemme}
If we choose $m:=\frac{t}{2\beta_n}$ with $t$ such that $\frac{t}{2\beta_n}\geq \frac{1}{2}\|x-\tx\|_2^2$, we can use Lemma \ref{lemme: aires multiples2} together with (\ref{eq: couplingRate}) to obtain the expected result. This ends the proof of Theorem \ref{thm: successfulGn}.
\end{proof}
We now proceed with the proof of Lemma \ref{lemme: aires multiples2}.
We use the results from the coupling on the Heisenberg group. 
In particular \emph{Lemma 3.4} from~\cite{banerjee2017coupling} can be written this way:
\begin{lemme}\label{AiresHmultiples}
	Let $\left((X_t,z_t),(\tX_t,\tz_t)\right)_t$ be a coupling of two Brownian motion on $\Ge_n$ such that:
	\begin{itemize}
	    \item $X_t^j=\tX_t^j$ for all $t$ and $j\geq 2$;
	    \item $dX_t^1=-d\tX_t^1$ on $[0,S]$ with $S=\inf\{t\geq0 | X_t^1=\tX_t^1\}$.
	\end{itemize}
  For all $m\geq \frac{1}{2}|X_0^1-\tX_0^1|^2$, we have: \begin{equation*}\esp\left[\left|\zeta^{1,j}_{S}-\zeta^{1,j}_{0}\right|\wedge {m}\right]\leq 
  2\sqrt{2m}|X_0^1-\tX_0^1|
  .\end{equation*}
\end{lemme}
The original proof can be found in~\cite{banerjee2017coupling} for $\zeta_S^{1,j}$ instead of $\zeta^{1,j}_{S}-\zeta^{1,j}_{0}$. This slight modification allows us to release the hypothesis on $m$ and to obtain better estimates for $C_1(n)$ and $C_2(n)$ in Theorem \ref{thm: successfulGn} (better order for $n$). 

We can now give the proof of Lemma \ref{lemme: aires multiples2}:
\begin{proof}[Proof of Lemma \ref{lemme: aires multiples2}]
We use the same notations as in  the proof of Theorem \ref{thm: successfulGn}. In particular we use the notation $\left({\ubbb}_t,{\utbbb}_t\right)_t$ to describe the couplings of Brownian motions induced by $\left(\uX_t,\utX_t\right)_t$ and $\uzeta_t$ such that $\ubbb_t^{-1}\utbbb_t=(\utX_t-X_t,\uzeta_t)$. Note that $\uzeta_t=P_h\zeta_tP_h^{t}$. As $\left(\uX_t,\utX_t\right)_t$ is defined by using the reflection coupling, the conditions of Lemma \ref{AiresHmultiples} are satisfied.

Note that $|\uX_0^1-\utX_0^1|^2=\|x-\tx\|_2^2$. Then for $m\geq  \frac{1}{2}\|x-\tx\|_2^2$ and for every $j\geq 2$, by applying Lemma \ref{AiresHmultiples}, we obtain:
\begin{equation}\label{eq: majorationZGn}
\esp\left[\frac{\left|\uzeta_{\tau_0}^{1,j}-\uzeta_{0}^{1,j}\right|}{m}\wedge 1\right]=\frac{1}{m}\esp\left[\left|\uzeta_{\tau_0}^{1,j}-\uzeta_{0}^{1,j}\right|\wedge m\right]\leq \frac{2\sqrt{2}}{\sqrt{m}}\|x-\tx\|_2.
\end{equation}
Moreover, for $t\in[0,\tau_0]$, $\uX^j_t=\utX^j_t$ for $2\leq j\leq n$, thus, $(\uzeta_{t})^{i,j}$ is constant on $[0,\tau_0]$ for all $2\leq i\neq j\leq n$. Using equivalences of norms and the fact that $\|\zeta_{\tau_0}-\zeta_0\|_2=\|\uzeta_{\tau_0}-\uzeta_0\|_2$, we get:
\begin{align}
\esp\left[\frac{\left\|\zeta_{\tau_0}\right\|_2}{{m}}\wedge 1\right]&\leq \frac{\|\zeta_0\|_2}{m}+\esp\left[\frac{\left\|\zeta_{\tau_0}-\zeta_{0}\right\|_2}{{m}}\wedge 1\right]=\frac{\|\zeta\|_2}{m}+\esp\left[\frac{\|\uzeta_{\tau_0}-\uzeta_0\|_2}{{m}}\wedge 1\right]\notag\\
&\leq\frac{\|\zeta\|_2}{m}+\sum\limits_{j=2}^n\esp\left[\frac{|\uzeta^{1,j}_{\tau_0}-\uzeta^{1,j}_0|}{{m}}\wedge 1\right]\notag\\
&\leq \frac{\|\zeta\|_2}{m}+\frac{2\sqrt{2}}{\sqrt{m}}(n-1)||x-\tx||_2\notag.
\end{align}
\end{proof}

Using the Coupling Inequality \eqref{eq: Chap1Aldous}, we directly obtain an estimate for Total Variation distance between the the Brownian motions:
\begin{cor}\label{Cor: dtvGnSucces}
Let $n\geq 2$, $g=(x,z)$, $\tg=(\tx,\tz)$ two points in $\Ge_n$ and $\zeta\in\mathfrak{so}(n)$ such that $g^{-1}\star \tg=(\tx-x,\zeta)$. For $t\geq \beta_n \|x-\tx\|^2_2$: 
\begin{equation}\label{eq: dTVGn}
    d_{TV}\left(\mu_t^{g},\mu_t^{\tg}\right)\le C_1(n)\frac{\|\tilde x- x\|_2}{\sqrt t}
    +C_2(n)\frac{\|\zeta\|_2}t.
\end{equation}
with $C_1(n)$ and $C_2(n)$ the constants given in Theorem \ref{thm: successfulGn}.
\end{cor}
\section{Coupling on homogeneous Carnot groups}\label{Sec: couplingAllHomogeneous}
\subsection{Homogeneous, step $2$ Carnot groups}\label{section: homogeneous}
We give here an overview of the Homogeneous Carnot groups. For more details, we refer the reader to \cite{bonfiglioli2007stratified}.

\begin{defi}\label{Prop: LoiPolynomiale}
    Set $n\geq 2$ and $m\leq\frac{n(n-1)}{2}$. Let $C^{(1)},\dots,C^{(m)}\in M_{n\times n}(\mathbb{R})$ such that the skew-symmetric parts $D^{(k)}=\frac{1}{2}\left(C^{(k)}-\left(C^{(k)}\right)^{t}\right)$ of $C^{(k)}$ are linearly independent for all $1\leq k\leq m$. We define $\Ge=(\mathbb{R}^{n+ m},\circ)$such that for all $(x,z):=(x_1,\dots,x_n,z_1,\dots z_m)$,  $(\tx,\tz):=(\tx_1,\dots,\tx_n,\tz_1,\dots \tz_m) \in \Ge$:
\begin{equation}\label{gp Carnot}
	(x,z)\circ(\tx,\tz)=\left(x+\tx,\left(z_k+\tz_k+\frac{1}{2}\langle C^{(k)}x|\tx\rangle\right)_{1\leq k\leq m}\right).
\end{equation} 
 Then $\Ge$ is called an homogeneous Carnot group with step $2$ and $n$ generators.  
\end{defi}

\begin{NB}\label{NB: isomorphisme Gn}
Let $n\geq 2$. For all $1\leq i<j\leq n$, let $C^{(i,j)}=\left(c_{l,r}^{(i,j)}\right)_{l,r}$ with  \begin{equation*}c_{l,r}^{(i,j)}=\left\{\begin{array}{ll}
	-1&\text{ if } (l,r)= (i,j)\\
	1 &\text{ if } (l,r)= (j,i)\\
	0&\text{ else }
\end{array}\right.\end{equation*}
As in Definition \ref{Prop: LoiPolynomiale}, these matrices define an homogeneous Carnot group $\left(\mathbb{R}^{n+ \frac{n(n-1)}{2}},\circ\right)$ which is isomorphic to the free, step $2$ Carnot group $\Ge_n$.
\end{NB}
Let $\Ge=(\mathbb{R}^{n+ m},\circ)$ be a homogeneous Carnot group with step $2$ and $n\geq 2$ generators. As for $G_n$, we can consider the horizontal left-invariant vector fields $(\bX_1,\hdots,\bX_n)$ such that $\bX_i(0)=\partial_{x_i}$ for all $1\leq i\leq n$. As the skew-symmetric matrices  $(D^{(k)})_{1\leq k\leq m}$ are linearly independent, for $\mathcal{H}:=\Span\{\bX_1,\hdots,\bX_n\}$, the Hörmander condition are satisfied. Still denoting by $d_{cc}$ the Carnot Carathéodory distance, we can then endow $\Ge$ with a structure of connected subRiemannian manifold.

The associated sub-Laplacian operator is also defined by $L=\sum\limits_{i=1}^n \bX_i^2$ and the associated Brownian motion 
satisfies the stochastic differential equation:
\begin{equation}\label{StratoMBCarnot2}
    d\bbb_t=\sum\limits_{i=1}^{n}\bX_i(\bbb_t)\circ dB_t^i
\end{equation} with $\circ d$ the Stratonovitch differential and $\left(B_t^1,\hdots,B_t^n\right)$ a Brownian motion on $\mathbb{R}^n$.
In particular the Brownian motion starting at $g=(x,z)\in \Ge$ can be defined as the continuous process $(\bbb_t)_t=\left(\Big(X_t,\big(z_{t}^k\big)_{1\leq k\leq m}\Big)\right)_t$
where:
\begin{itemize}
    \item $(X_t)_t=(X_t^1,\hdots,X_t^n)_t$ is a $\mathbb{R}^n$-Brownian motion starting from $x$;
    \item For all $1\leq k\leq m$, $dz_{t}^k=\frac{1}{2}\sum\limits_{i,j=1}^nc_{i,j}^{(k)}X_t^jdX_t^i+\frac{1}{4} \sum\limits_{i=1}^n c_{i,i}^{(k)}dt$.
    Still denoting  by $\mathcal{A}_t^{i,j}$ the Lévy area associated to $(X_t^i,X_t^j)$, we can also write: \begin{align*}z_t^k=z_0^k+\sum\limits_{1\leq i<j\leq n}\frac{1}{2}(c_{j,i}^{(k)}-c_{i,j}^{(k)})\mathcal{A}_t^{i,j}+\frac{1}{4}\sum\limits_{i,j=1}^nc_{i,j}^{(k)}\left(X_{t}^i X_{t}^j-X_{0}^i X_{0}^j\right)
    \end{align*}
    
\end{itemize}

The following proposition allows us to reduce our study to free, step $2$ Carnot groups:
\begin{prop}\label{epimorphisme}
Let $\Ge=\left(\mathbb{R}^{n+ m},\circ\right)$ be a homogeneous Carnot group with step $2$ and $n$ generators. As previously, ${\Ge}_n$ denotes the free, step $2$ Carnot group with $n$ generators. There exists a smooth surjective Lie group morphism $\phi:{\Ge}_n\to \Ge$ such that:
\begin{enumerate}
\item\label{item0} $\phi$ preserves the horizontal coordinate "$x$";
    \item\label{item1} for all $g,\tg\in \Ge_n$, $d_{cc}(g,\tg)\geq d_{cc}(\phi(g),\phi(\tg))$;
    \item\label{item2} for all $a,\tilde{a}\in \Ge$, there exist $g,\tg\in \Ge_n$ such that $\phi(g)=a$, $\phi(\tg)=\tilde{a}$ and $d_{cc}(g,\tg)=d_{cc}(a,\tilde{a})$;
    \item\label{item3} if $(\bbb_t)_t$ is a Brownian motion on $\Ge_n$, then $\left(\phi(\bbb_t)\right)_t$ is a Brownian motion on $\Ge$.
\end{enumerate}
\end{prop}
\begin{proof}
  During this proof we denote by $(\bX_1,\hdots,\bX_n)$ the horizontal left-invariant vector fields on $\Ge_n$ and by $(\tilde{\bX}_1,\hdots,\tilde{\bX}_n)$ the horizontal left-invariant vector fields on $\Ge$. It is well known that a surjective Lie group morphism $\phi:\Ge_n\to \Ge$ preserving the horizontal coordinate can be constructed such that $\tilde{\bX}_i(\gamma\circ\phi)=\left(\bX_i (\gamma)\right)\circ\phi$ (see \cite{bonfiglioli2007stratified,DPS19}) ($\Ge$ can be lifted from $\Ge_n$). In particular, by considering some convenient horizontal path, we obtain the two first points of Proposition \ref{epimorphisme} (see \emph{Lemma 6.1.} in~\cite{DPS19} for more details).
  
   The last point is obtained by considering the stochastic equation (\ref{StratoMBCarnot})  satisfied by the Brownian motions on $\Ge_n$. As, the equation is given with the Stratonovitch integral and $\ph$ is smooth, we get:
	\begin{equation*}
	    d\ph(\bbb_t)= \left(d\ph\cdot\sum\limits_{i=1}^n\bX_i\right)(\ph(\bbb_t))\circ dB_t^i=\sum\limits_{i=1}^nd\tilde{\bar{X}}_i(\ph(\bbb_t))\circ dB_t^i.
	\end{equation*}
	Then $(\ph(\bbb_t))_t$ satisfies Equation (\ref{StratoMBCarnot2}) for $\Ge$.
\end{proof}
\subsection{Coupling results}\label{Subsec: couplingAllHomogeneous}
From Theorem \ref{thm: successfulGn}, we can obtain a successful coupling of Brownian motions for any homogeneous step $2$ Carnot group $\Ge$.
\begin{cor}\label{cor: couplingAllHomogeneous}
Let $\Ge$ be an homogeneous step $2$ Carnot group of rank $n$. For all $g=(x,z),\tg=(\tx,\tz)\in \Ge$ there exists a successful coupling $(\bbb_t,\tbbb_t)$ of Brownian motions starting from $(g,\tg)$ such that, for all $t\geq 4\beta_n\|x-\tx\|_2^2$:
\begin{align}\label{eq: couplingRateAllHomogeneous}
    \pr(\tau>t)\leq &C_1(n)\frac{\|x-\tx\|_2}{\sqrt{t}}\mathbb{1}_{\{x\neq\tx\}}+\frac{C_2(n)}{{m_1(n)^2}
    }\frac{d_{cc}(g,\tg)^2}{t}\\
    \leq &C_1(n)\frac{d_{cc}(g,\tg)}{\sqrt{t}}\mathbb{1}_{\{x\neq\tx\}}+\frac{C_2(n)}{{m_1(n)^2}
    }\frac{d_{cc}(g,\tg)^2}{t}
\end{align}
with $C_1(n)$ and $C_2(n)$ the explicit constants from Theorem \ref{thm: successfulGn} {and $m_1(n)$ the constant from \eqref{eq: equivSeminorme}}.
\end{cor}
\begin{proof}
Let $g=(x,z)$, $\tg=(\tx,\tz)\in \Ge$. Using Proposition \ref{epimorphisme}, we can build a surjective morphism $\phi:\Ge_n\to G$ that preserves the horizontal coordinates. Moreover, there exists $a=(x,v),\tilde{a}=(\tx,\tilde v)\in \Ge_n$ such that:
\begin{equation}\label{egaliteDistance}\phi(a)=g,\  \phi(\tilde{a})=\tilde{g}\text{ and } d_{cc}(g,\tg)=d_{cc}(a,\tilde{a}).
\end{equation}
Using the construction from Theorem \ref{thm: successfulGn}, we construct a successful coupling $(\mathbf{B}_t,\tilde{\mathbf{B}}_t)_t$ of Brownian motions starting from $(a,\tilde{a})$ on $\Ge_n$. Denoting by $S$ the first coupling time of $(\mathbf{B}_t,\tilde{\mathbf{B}}_t)_t$ we have for $t\geq4\beta_n\|x-\tx\|_2^2$:
\begin{equation*}
    \pr(S>t)\leq C_1(n)\frac{\|x-\tx\|_2}{\sqrt{t}}\mathbb{1}_{\{x\neq \tx\}}+ C_2(n)\frac{\|\zeta\|_2}{t}.
\end{equation*}
On $\Ge_n$, using the equivalence of the homogeneous norm (\ref{eq: equivSeminorme})
we have: $\|\zeta\|_2\leq \frac {1}{{m_1(n)}^2}
d_{cc}(a,\tilde{a})^2$
Thus:
\begin{equation}\label{eq: couplingRateHomogeneous}
    \pr(S>t)\leq C_1(n)\frac{\|x-\tx\|_2}{\sqrt{t}}\mathbb{1}_{\{x\neq \tx\}}+ \frac{C_2(n)}{{m_1(n)^2}
    } \frac{d_{cc}(a,\tilde{a})^2}{t}.
\end{equation}
Moreover, from Remark \ref{NB: distanceHorizontale}, we have $\|x-\tx\|_2\leq d_{cc}(a,\tilde a)$.

Then, $\left(\phi(\mathbf{B}_t),\phi(\tilde{\mathbf{B}}_t)\right)_t$ defines a coupling of Brownian motions on $\Ge$ starting at $(g,\tg)$. Denoting by $\tau$ its first coupling time, we have $\pr(S>t)\leq \pr(\tau>t)$. Using (\ref{eq: couplingRateHomogeneous}) together with (\ref{egaliteDistance}), we obtain the expected result.
\end{proof}

As for Corollary \ref{Cor: dtvGnSucces}, by applying the Coupling Inequality (\ref{eq: Chap1Aldous}), we directly obtain an estimate for the Total Variation distance between the laws of the Brownian motion.
Denote by $\mu_t^{(x,z)}$ the law at $t$ of the subRiemannian Brownian motion started at $(x,z)$.
\begin{cor}\label{Cor: dtvGnSucces2}
Let $\Ge$ be an homogeneous, step $2$ Carnot group of rank $n$, $n\geq 2$. For any  $g=(x,z),\tg=(\tx,z)\in \Ge$ and for $t\geq \beta_n \|x-\tx\|^2_2$: 
\begin{equation*}
        d_{TV}\left(\mu_t^{g},\mu_t^{\tg}\right)\le C_1(n)\frac{d_{cc}(g,\tg)}{\sqrt t}
    +\frac{C_2(n)}{{m_1(n)^2}
    }\frac{d_{cc}(g,\tg)^2}t.
\end{equation*}
\end{cor}

\section{Comparison between the first coupling time and the first exit time from a domain}\label{sec: coupling Vs Exit}
In this section, we consider a domain $D$ on $\Ge_n$. For any process $(Y_t)_t$ on $\Ge_n$ we define $\tau_D(Y):=\inf\{t\geq0\ | \ Y_t\notin D\}$. Using the pseudo-distance $\bd$ introduced in Subsection \ref{subsec: pseudo-distance}, for any $g\in D$, we denote $\delta_g=\bd(g,D^c)$. For the coupling $(\bbb_t,\tbbb_t)_t$ constructed in Theorem \ref{thm: successfulGn}, we are going to compare the coupling time $\tau$ with $\tau_D(\bbb)\wedge \tau_D(\tbbb)$ in order to obtain gradient inequalities as in~\cite{CranstonSoblogRn,CranstonRefl,banerjee2017coupling}. As previously we adapt the proof used for the Heisenberg group to all free, step $2$ Carnot groups~\cite{banerjee2017coupling}.
\subsection{Main results}
For $n\geq 2$ and $g=(x,z)$, $\tg=(\tx,\tz)\in \Ge_n$ and $\zeta\in\mathfrak{so}(n)$ such that $g^{-1}\star \tg=(\tx-x,\zeta)$, we define $\hx:=\frac{x+\tx}{2}\in\mathbb{R}^n$ and $\hz:=\frac{z+\tz}{2}\in\mathfrak{so}(n)$. In particular, $x-\hx=\frac{x-\tx}{2}$ and $z-\hz-\frac{1}{2}\hx\symp x=\frac{\zeta}{2}$. For any $\alpha>0$ and $\gamma>0$ real, we also define the set $Q(\alpha,\gamma)\subset \Ge_n$ such that:
\begin{equation}\label{eq: hypercube}
    Q(\alpha,\gamma):=\{(y,v)\in \Ge_n\ | \ \|y-\hx\|_2< \alpha, \|v-\hz-\frac{1}{2}\hx\symp y\|_2< \gamma^2\}.
\end{equation}

Consider $(\bbb_t,\tbbb_t)_t$ the coupling starting at $(g,\tg)$ described in Theorem \ref{thm: successfulGn}. We are first dealing with the problem for $D=Q(\alpha,\gamma)$. We have the following result:
\begin{prop}\label{prop: exitTimeCube}
Suppose that $\|x-\tx\|_2\leq \frac{2}{\sqrt{n-1}}$. Then, for any $0<\alpha\leq\gamma$, there exist some constants $D_1$ and $D_2$, that depend neither on $\alpha,\gamma$ nor on the starting points of the coupling, such that:
    \begin{align*}
    \pr\left(\tau>\tau_{Q(\alpha,\gamma)}(\bbb)\right)&\leq D_1(n,\gamma,\alpha)\|x-\tx\|_2+D_2(n,\gamma,\alpha)\max\left(
    \|\zeta\|_2,\|\zeta\|_2^2\right)
    \end{align*}
    with:
    \begin{align*}
        D_1(n,\gamma,\alpha)&\leq D_1\left(n\sqrt{n}+\frac{n\sqrt{n}}{\alpha}+\frac{\alpha^3}{\gamma^4}n+\frac{{n^{7+\frac{1}{4}}b_n^2}}{\gamma^4}+\frac{{n^{9+\frac{1}{4}}b_n^2}}{\alpha^4}\right);\\
        D_2(n,\gamma,\alpha)&\leq D_2 \left(1+\frac{{n^{6+\frac{1}{2}}b_n^2}}{\gamma^4}+\frac{{n^{8+\frac{1}{2}}b_n^2}}{\alpha^4}\right);\\
        b_n&{\leq 2\sqrt{2\pi n}.}
    \end{align*}
    The same inequality is true for $\pr\left(\tau>\tau_{Q(\alpha,\gamma)}(\tbbb)\right)$.
\end{prop}
Proposition \ref{prop: exitTimeCube} will be proven in Subsection \ref{subsec: preuveExit time}. For now we give a direct consequence for the exit time of any domain $D$.
\begin{thm}\label{thm: sortieDomaine}
Let $g=(x,z),\tg=(\tx,\tz)\in \Ge_n$ and $g^{-1}\star\tg=(\tx-x,\zeta)$ such that $\bd(g,\tg)\leq\frac{\sqrt{2}\delta_g}{3}$ and $\|x-\tx\|_2\leq \frac{2}{\sqrt{n-1}}$
. Denoting by $(\bbb_t,\tbbb_t)_t$ the coupling starting at $(g,\tg)$ described in Theorem \ref{thm: successfulGn}, we have:
\begin{align*}
    \pr\left(\tau>\tau_D(\bbb)\wedge\tau_D(\tbbb)\right)&\leq {D_1}\times\left(n\sqrt{n}+\frac{n\sqrt{n}}{\delta_g}+\frac{{n^{10+\frac{1}{4}}}}{\delta_g^4}\right)\|x-\tx\|_2\\
    &+{D_2}\times \left(1+\frac{{n^{9+\frac{1}{2}}}}{\delta_g^4}\right)\max\left(
    \|\zeta\|_2,\|\zeta\|_2^2\right)
\end{align*}
with ${D}_1$ and ${D_2}$ two constant that depend neither on $g,\tg$ nor on $n$.
\end{thm}
\begin{proof}[Proof of Theorem \ref{thm: sortieDomaine}]
By construction of $(\hx,\hz)$, we have:
\begin{equation*}\bd\left((\hx,\hz),(x,z)\right)=\bd\left((\hx,\hz),(\tx,\tz)\right)\leq\frac{1}{\sqrt{2}}\bd\left((\tx,\tz),(x,z)\right)\leq \frac{\delta_g}{3}.
\end{equation*}
Using the pseudo triangular inequality (\ref{eq: pseudo-triangle}), for any $(y,v)\in \Ge_n$, we have:
\begin{equation}
\bd\left((y,v),(x,z)\right)\leq {\frac{3}{2}}\Big(\bd\left((y,v),(\hx,\hz)\right)+\bd\left((\hx,\hz),(x,z)\right)\Big).
\end{equation}
Thus, if $(y,v)\in Q\left(\frac{\delta_g}{3\sqrt{2}},\frac{\delta_g}{3\sqrt{2}}\right)$, then $\bd\left((y,v),(x,z)\right)< {\frac{3}{2}}\left(\sqrt{2}\times\frac{\delta_g}{3\sqrt{2}}+\frac{\delta_g}{3}\right)= \delta_g$.
Thus, $Q:=Q\left(\frac{\delta_g}{3\sqrt{2}},\frac{\delta_g}{3\sqrt{2}}\right)\subset D$.
We have:
\begin{align*}
     \pr\left(\tau>\tau_D(\bbb)\wedge\tau_D(\tbbb)\right)&\leq  \pr\left(\tau>\tau_Q(\bbb)\wedge\tau_Q(\tbbb)\right)\\
     &\leq  \pr\left(\tau>\tau_Q(\bbb)\right)+\pr\left(\tau>\tau_Q(\tbbb)\right).
\end{align*}
As $\|x-\tx\|_2\leq \frac{2}{\sqrt{n-1}}$, we can apply Proposition \ref{prop: exitTimeCube}. We get the expected result.
\end{proof}

\subsection{Proof of Proposition \ref{prop: exitTimeCube}}\label{subsec: preuveExit time}
\begin{proof}[Proof of Proposition \ref{prop: exitTimeCube}]
It is enough to prove the proposition for $(\bbb_t)_t$.
We first notice that 
\begin{equation*}
    \tau_{Q(\alpha,\gamma)}(\bbb)=\inf\{t>0\ |\ \|X_t-\hx\|_2\geq \alpha\text{ or }\|U_t\|_2\geq \gamma^2\}
\end{equation*} with $U_t:=z_t-\hz-\frac{1}{2}\hx\symp X_t$. In particular, we have $dU_t=\frac{1}{2}(X_t-\hx)\symp dX_t$ and $U_0=\frac{z-\tz}{2}-\frac{1}{4}x\symp\tx=\frac{\zeta}{2}$.
We use the notations from the the proofs of Section \ref{Section:Brownian BridgeGn}:
\begin{itemize}
    \item As in the proof of Theorem \ref{thm: successfulGn}, we denote: $\tau_0:=\inf\{t\geq0\ |\ X_t=\tX_t\}$, in particular on $[0,\tau_0]$, the processes $(X_t)_{t\in[0,\tau_0]}$ and $(\tX_t)_{t\in[0,\tau_0]}$ are coupled by reflection.
    \item As in the proof of Theorem \ref{Couplage par lignes}, we denote for $1\leq i\leq n-1$:
\begin{equation*} 
\tau_{i}:=\inf\left\{t>\tau_{i-1}\ | \  \uX_t=\utX_t \text{ and }\uzeta_t^{i,j}=\tilde{\uzeta}_t^{i,j} \ \forall j\in\{1,\hdots n\}\setminus\{i\}\right\}.
\end{equation*}
\end{itemize}
We remind that $\tau:=\inf\{t>0\ | \ \bbb_t=\tbbb_t\}$ is exactly equal to $\tau_{n-1}$. We also denote $\tau_{-1}:=0$. As we have a different description of the coupling on each of the intervals of time $[0,\tau_0]$, $[\tau_0,\tau_1]$, ..., $[\tau_{n-2},\tau_{n-1}]$, it will be convenient to decompose the process $(\bbb_t)_t=(X_t,z_t)_t$ following this time decomposition. One of the main tool that will be used is the Burkholder-Davis-Gundy Inequality that we will denote afterward "B-D-G Inequality". Then it will be convenient to consider martingales starting from $0$ on each of these intervals of time. Keeping these constraints in mind, we define:
\begin{itemize}
    \item for $t\in [0,\tau_0]$, $Y_0(t):=X_t-\hx$ and $V_0(t):=U_t$;
    \item for $1\leq i\leq n-1$ and $t\in[\tau_{i-1},\tau_{i}]$, $Y_i(t):=X_t-X_{\tau_{i-1}}$ and $V_i(t):=\frac{1}{2}\int_{\tau_{i-1}}^t (X_s-X_{\tau_{i-1}}) \symp dX_s\in\mathfrak{so}(n)$. In particular, we have: \begin{equation}\label{V_i}
        V_i^{j,l}(t)=\frac{1}{2}\left(\int_{\tau_{i-1}}^t\left(X_s^j-X_{\tau_{i-1}}^j\right)dX_s^l-\int_{\tau_{i-1}}^t\left(X_s^l-X_{\tau_{i-1}}^l\right)dX_s^j\right).
    \end{equation}
\end{itemize}
Proceeding by induction one can then check that, for all $0\leq i\leq n-1$ and $t\in[\tau_{i-1},\tau_{i}]$, we have:
\begin{equation}\label{eq: decompositionQ}
    (X_t-\hx,U_t)=\big(Y_0(\tau_0),V_0(\tau_0)\big)\star\hdots\star \big(Y_{i-1}(\tau_{i-1}),V_{i-1}(\tau_{i-1})\big)\star\big(Y_{i}(t),V_{i}(t)\big).
\end{equation}
Moreover, using the fact that we iterate $i+1$ times the operation $\star$, we have for all $1\leq i\leq n-1$:
\begin{align}\label{eq: decompositionQ2}
    (X_t-\hx,U_t)&=\left(\sum\limits_{j=0}^{i-1}Y_j(\tau_j)+Y_{i}(t),\right.\notag\\
    &\left.\sum\limits_{j=0}^{i-1}V_j(\tau_j)+V_{i}(t)+\frac{1}{2}\sum\limits_{j=0}^{i-1} Y_j(\tau_j)\symp\left(\sum\limits_{l=j+1}^{i-1}Y_l(\tau_l)+Y_{i}(t)\right)\right).
\end{align}

 We consider two sequences $(\alpha(i))_{0\leq i\leq n-1}$ and $(\gamma(i))_{0\leq i\leq n-1}$ satisfying:
 \begin{equation}\label{eq: ConditionsSuites}
 \sum\limits_{i=0}^{n-1}\alpha(i)<\alpha\text{ and } \sum\limits_{i=0}^{n-1}\gamma(i)^2+\frac{1}{2}\sum\limits_{0\leq i<j\leq n-1 }\alpha(i)\alpha(j)<\gamma^2.\end{equation} 

 For all $0\leq i\leq n-1$, we define the events:
\begin{equation*}
    A_i:=\{\omega, \ \sup\limits_{\tau_{i-1}\leq t\leq \tau_{i}}\|Y_i(t)\|_2> \alpha(i)\}\text{ and }
   \Gamma_i:=\{\omega,\ \sup\limits_{\tau_{i-1}\leq t\leq \tau_{i}}\|V_i(t)\|_2> \gamma(i)^2\}.
\end{equation*} 
Using (\ref{eq: decompositionQ2}), we remark that under the event $\bigcap\limits_{i=0}^{n-1}\left(A_i^c\cap \Gamma_i^c\right)$, we have for all $0\leq i\leq n-1$ and for all $t\in [\tau_{i-1},\tau_i]$:
\begin{equation}
    \|X_t-\hx\|_2\leq \sum\limits_{j=0}^{i}\alpha(j)<\alpha\text{ and }
    \|U_t\|_2\leq\sum\limits_{j=0}^{i}\gamma(j)^2+\frac{1}{2}\sum\limits_{0\leq j<l\leq i }\alpha(j)\alpha(l)<\gamma^2.
\end{equation}
In other words, $\bbb_t\in Q(\alpha,\gamma)$ for all $t\in[0,\tau]$.
Thus $\pr\left(\tau>\tau_{Q(\alpha,\gamma)}(\bbb)\right)\leq \pr\left(\bigcup\limits_{i=0}^{n-1}\left(A_i \cup \Gamma_i\right)\right)$.

To study these probabilities we will use the three following lemmas. 

The first one contains results on the reflection coupling in $\mathbb{R}^n$ and is due to Cranston. These results can be found in the proof of \emph{Theorem 1} in~\cite{CranstonSoblogRn}. Using the Euclidean norm $\|\cdot\|_2$ instead of the maximum norm $\|\cdot\|_{\infty}$ the results can be presented as follows:
\begin{lemme}\label{Lemme CranstonSortieRef}
For any $a>0$, denote $S(a):=\inf\{t\geq0 \ |\ \|X_t-\hx\|_{2}> a\}$. With the previous notations:
\begin{equation*}
    \pr\left(\tau_0>S(a)\right)\leq \frac{n\sqrt{n}\|x-\tx\|_2}{2a} \text{ and } \esp\left[\tau_0\wedge S(a)\right]\leq \frac{\|x-\tx\|_2}{2}a.
\end{equation*}
\end{lemme}
The second lemma can be seen as the generalisation in higher dimension of \emph{Lemma 4.1} from~\cite{banerjee2017coupling}:
\begin{lemme}\label{lemme: BDGNonCoAdapte}
For any $1\leq i\leq n-1$ and $M>0$ there exists a constant $C$ (depending neither on the starting points of the processes, nor on the rank $n$ of the group) such that:
\begin{align}
    \esp\left[\mathbb{1}_{\{\tau_i-\tau_{i-1}<M\}}\sup\limits_{\tau_{i-1}\leq t \leq\tau_i}\|Y_i(t)\|_2^4\right]&\leq Cn^2\esp\left[\left(\tau_{i}-\tau_{i-1}\right)^2\wedge M^2 \right];\label{eq: MajY}\\
   \esp\left[\mathbb{1}_{\{\tau_i-\tau_{i-1}<M\}}\sup\limits_{\tau_{i-1}\leq t \leq\tau_i}\|V_i(t)\|_2^2\right]&\leq Cn(n-1)\esp\left[\left(\tau_{i}-\tau_{i-1}\right)^2\wedge M^2 \right]\label{eq: MajV}.
\end{align}
\end{lemme}
The third lemma gives an estimate for the coupling time after time $\tau_0$:
\begin{lemme}\label{lemme Esp temps de couplage}
Using the previous notations,
for any $M\geq \frac{\beta_n}{2}\|x-\tx\|_2^2$:
\begin{align*}
    \esp\left[(\tau-\tau_0)^2\wedge M^2\right]&\leq 
    \frac{4}{3}M^{\frac{3}{2}}2\sqrt{2\beta_n}(n-1)||x-\tx||_2+2M\beta_n \|\zeta\|_2.
    \end{align*}
\end{lemme}

Please note that Lemmas \ref{lemme: BDGNonCoAdapte} and \ref{lemme Esp temps de couplage} will be proven in Subsection \ref{subsec: BDGNonCoAdapte}. For now, we pursue the proof of Proposition \ref{prop: exitTimeCube}.

From Lemma \ref{Lemme CranstonSortieRef}, we obtain:
\begin{equation}\label{eq: A_0 1}
    \pr(A_0)\leq \pr\big(\tau_0>S\left(\alpha(0)\right)\big)\leq \frac{n\sqrt{n}\|x-\tx\|_2}{2\alpha(0)}.
\end{equation}
Using the Markov Inequality, we also have:
\begin{align*}
    \pr(A_0^c\cap \Gamma_0)&\leq\pr\left(\sup\limits_{0\leq t\leq \tau_0\wedge S\left(\alpha(0)\right)}\|U_t\|_2>\gamma(0)^2\right)\leq \frac{1}{\gamma(0)^4}\esp\left[\sup\limits_{0\leq t\leq \tau_0\wedge S\left(\alpha(0)\right)}\|U_t\|_2^2\right]\\
    &\leq \frac{1}{\gamma(0)^4}\left(\esp\left[\sup\limits_{0\leq t\leq \tau_0\wedge S\left(\alpha(0)\right)}\|U_t-U_0\|_2^2\right]+\frac{\left\|\zeta\right\|^2_2}{4}\right).
\end{align*}
We remind that, for $1\leq j\neq l\leq n$, $dU_t^{j,l}=\frac{1}{2}\left((X_t-\hx)^jdX_t^l-(X_t-\hx)^ldX_t^j\right)$. As the reflection coupling is a co-adapted coupling, $\tau_0\wedge S(\alpha(0))$ is a stopping time for the common filtration of $(X_t)_t$ and $(\tX_t)_t$. Using the B-D-G Inequality, we obtain:
\begin{align*}
    \sum\limits_{1\leq j<l\leq n}\esp\bigg[\sup\limits_{0\leq t\leq \tau_0\wedge S(\alpha(0))}&\left.\left(U_t^{j,l}-U_0^{j,l}\right)^2\right]\\
    &\leq \sum\limits_{1\leq j<l\leq n}\frac{c_2}{4} \esp\left[\int_0^{\tau_0\wedge S(\alpha(0))}|X_t^j-\hx^j|^2+|X_t^l-\hx^l|^2dt\right]\\
    &\leq \frac{c_2(n-1)}{4} \esp\left[\int_0^{\tau_0\wedge S(\alpha(0))}\|X_t-\hx\|_2^2dt\right]\\
    &\leq \frac{c_2(n-1)}{4}\alpha(0)^2 \esp\big[\tau_0\wedge S(\alpha(0))\big]\\
    &\leq \frac{c_2(n-1)}{8}\alpha(0)^3\|x-\tx\|_2.
\end{align*}
Finally we have:
\begin{equation}\label{eq: A_0 2}
    \pr\left(A_0^c\cap\Gamma_0\right)\leq \frac{c_2n}{8}\frac{\alpha(0)^3}{\gamma(0)^4}\|x-\tx\|_2+\frac{\|\zeta\|_2^2}{4\gamma(0)^4}.
\end{equation}
We now look at $\pr\left(\bigcup\limits_{i=1}^{n-1}(A_i\cup \Gamma_i)\right)$. We have:
\begin{align*}
    \pr\left(\bigcup\limits_{i=1}^{n-1}(A_i\cup \Gamma_i)\right)&\leq \sum\limits_{i=1}^{n-1}\pr\left(\tau_i-\tau_{i-1}>\frac{\beta_n}{n-1}\right)\notag\\
    &+\sum\limits_{i=1}^{n-1}\pr\left({\tau_i-\tau_{i-1}\leq\frac{\beta_n}{n-1}},\sup\limits_{\tau_{i-1}\leq t\leq \tau_i}\|Y_i(t)\|_2>\alpha(i)\right)\notag\\
    &+\sum\limits_{i=1}^{n-1}\pr\left({\tau_i-\tau_{i-1}\leq\frac{\beta_n}{n-1}},\sup\limits_{\tau_{i-1}\leq t\leq \tau_i}\|V_i(t)\|_2>\alpha(i)\right).
\end{align*} 

As in Subsection \ref{subsec: memefibreGn}, for all $1\leq i\leq n-1$, we denote $v_{\tau_0}^i=(\zeta_{\tau_0}^{i,j})_{j\in\{1,\hdots,n\}\setminus\{ i\}}\in\mathbb{R}^{n-1}$. We also recall that $\beta_n=(n-1)^{\frac{3}{2}}b_n$.
Iterating $(n-1)$-times Proposition \ref{Brique1} together with Remark \ref{Rem: Brique}, we have:
\begin{equation*}\sum\limits_{i=1}^{n-1}\pr\left(\tau_i-\tau_{i-1}>\frac{\beta_n}{n-1}\right)\leq\sum\limits_{i=1}^{n-1}\esp\left[\left(\|v_{\tau_{i-1}^i}\|_2\frac{(n-1)b_n}{\beta_n}\right)\wedge 1\right]\leq \esp\left[\|\zeta_{\tau_0}\|_2\wedge (n-1)\right].
\end{equation*}

By hypothesis, we have $\frac{1}{2}\|x-\tx\|_2^2\leq n-1$. Thus, using Lemma \ref{lemme: aires multiples2}, we get:
\begin{equation}\label{eq: estimateTpsSortiePont}
\sum\limits_{i=1}^{n-1}\pr\left(\tau_i-\tau_{i-1}>\frac{\beta_n}{n-1}\right)\leq\|\zeta\|_2+2\sqrt{2}(n-1)^{\frac{3}{2}}||x-\tx||_2.
\end{equation}

Using Markov's Inequality and Lemma \ref{lemme: BDGNonCoAdapte}, we obtain:
\begin{align*}
    \pr\left({\tau_i-\tau_{i-1}\leq\frac{\beta_n}{n-1}},\sup\limits_{\tau_{i-1}\leq t\leq \tau_i}\|Y_i(t)\|_2>\alpha(i)\right)&\leq\frac{Cn^2}{\alpha(i)^4}\esp\left[\left(\tau_i-\tau_{i-1}\right)^2\wedge\left(\frac{\beta_n}{n-1}\right)^2\right] \\
    \pr\left({\tau_i-\tau_{i-1}\leq\frac{\beta_n}{n-1}},\sup\limits_{\tau_{i-1}\leq t\leq \tau_i}\|V_i(t)\|_2>\gamma(i)^2\right)&\leq\frac{Cn(n-1)}{\gamma(i)^4}\esp\left[\left(\tau_i-\tau_{i-1}\right)^2\wedge\left(\frac{\beta_n}{n-1}\right)^2\right].
\end{align*}

Thus:
\begin{align}\label{eq: Pendant le pont}
\pr\left(\bigcup\limits_{i=1}^{n-1}(A_i\cup \Gamma_i)\right)&\leq \|\zeta\|_2+2\sqrt{2}(n-1)^{ \frac{3}{2}}||x-\tx||_2\notag\\
&+C\left(\frac{n^2}{\min\limits_{1\leq i\leq n-1}(\alpha(i))^4}+\frac{n(n-1)}{\min\limits_{1\leq i\leq n-1}(\gamma(i))^4}\right)\esp\left[\left(\tau-\tau_{0}\right)^2\wedge \frac{\beta_n^2}{n-1}\right].
\end{align}

Still using the fact that $\|x-\tx\|_2^2\leq \frac{2}{\sqrt{n-1}}$ we can also use Lemma \ref{lemme Esp temps de couplage} with $M=\frac{\beta_n}{\sqrt{n-1}}$ to get:
\begin{equation}\label{eq: estimateTpsSortiePont2}
    \esp\left[\left(\tau-\tau_{0}\right)^2\wedge \frac{\beta_n^2}{n-1}\right]\leq 
    \frac{8\sqrt{2}}{3}\beta_n^{2}(n-1)^{\frac{1}{4}}||x-\tx||_2+\frac{2\beta_n^2}{\sqrt{n-1}} \|\zeta\|_2.
    \end{equation}

 To construct two sequences $(\alpha(i))_{0\leq i\leq n-1}$ and $(\gamma(i))_{0\leq i\leq n-1}$ satisfying (\ref{eq: ConditionsSuites}), we can choose $\alpha(0)=\frac{\alpha}{2}$, $\gamma(0)=\frac{\gamma}{\sqrt{2}}$ and, for $1\leq i\leq n-1$, $\alpha(i)=\frac{\alpha}{4(n-1)}$ and $\gamma(i)=\frac{\gamma}{2\sqrt{n-1}}$. Indeed, we have: \begin{equation*}
\sum\limits_{i=0}^{n-1}\alpha(i)<\frac{\alpha}{2}+\sum\limits_{i=1}^{n-1}\frac{\alpha}{4(n-1)}<\alpha
\end{equation*} 
and 
\begin{align*}
\sum\limits_{i=0}^{n-1}\gamma(i)^2&+\frac{1}{2}\sum\limits_{0\leq i<j\leq n-1 }\alpha(i)\alpha(j)\\
&\leq \frac{\gamma^2}{2}+\sum\limits_{i=1}^{n-1}\frac{\gamma^2}{4(n-1)}+\frac{1}{2}\left(\frac{\alpha}{2}\sum\limits_{j=1}^{n-1}\frac{\alpha}{4(n-1)}+\sum\limits_{1\leq i<j\leq n-1 }\frac{\alpha^2}{16(n-1)^2}\right)\\
&<\frac{\gamma^2}{2}+\frac{\gamma^2}{4}+\frac{\alpha^2}{16}+\frac{\alpha^2}{32}<\gamma^2.
\end{align*} 
The last inequality follows from the hypothesis $\alpha\leq \gamma$.
We thus obtain inequalities  (\ref{eq: A_0 2}), (\ref{eq: Pendant le pont}) and (\ref{eq: estimateTpsSortiePont2}) which ends the proof of Proposition \ref{prop: exitTimeCube}.
\end{proof}
\subsection{Proofs of the Lemmas}\label{subsec: BDGNonCoAdapte}
In this subsection we give the announced proofs for Lemma \ref{lemme: BDGNonCoAdapte} and Lemma \ref{lemme Esp temps de couplage}.

We begin with the proof of Lemma \ref{lemme Esp temps de couplage}.
\begin{proof}[Proof of Lemma \ref{lemme Esp temps de couplage}]
From Theorem \ref{Couplage par lignes} and Lemma \ref{lemme: aires multiples2}, for $\sqrt{t}\geq m_0:=\frac{\beta_n}{2}\|x-\tx\|_2^2$ we get:
\begin{align*}
    \pr\left(\tau-\tau_0>\sqrt{t}\right)&\leq \esp\left[\left(\|\zeta_{\tau_0}\|_2\frac{\beta_n}{\sqrt{t}}\right)\wedge 1\right]\\
    &\leq t^{-\frac{1}{2}} \beta_n\|{\zeta}\|_2+t^{-\frac{1}{4}}2\sqrt{2\beta_n}(n-1)||x-\tx||_2.
\end{align*}
Then:
\begin{align*}
    \esp&\left[(\tau-\tau_0)^2\wedge M^2\right]=\int_0^{m_0^2} \pr((\tau-\tau_0)^2>t)dt+\int_{m_0^2}^{M^2} \pr((\tau-\tau_0)^2>t)dt\\
    &\leq{m_0^2} +\int_{m_0^2}^{M^2} \pr\left(\tau-\tau_0>\sqrt{t}\right)dt\\
    &\leq{m_0^2}+ \int_{m_0^2}^{M^2}t^{-\frac{1}{4}}dt\times2\sqrt{2\beta_n}(n-1)||x-\tx||_2+\int_{0}^{M^2}t^{-\frac{1}{2}}dt\times\beta_n\|{\zeta}\|_2\\
    &={m_0^2}+\frac{4}{3}\left(M^{\frac{3}{2}}-m_0^{\frac{3}{2}}\right)2\sqrt{2\beta_n}(n-1)||x-\tx||_2+2M\beta_n\|{\zeta}\|_2.
\end{align*}
As ${m_0^2}-\frac{4}{3}m_0^{\frac{3}{2}}2\sqrt{2\beta_n}(n-1)||x-\tx||_2=m_0^{\frac{3}{2}}\sqrt{\frac{\beta_n}{2}}\|x-\tx\|_2\left(1-\frac{16}{3}(n-1)\right)<0$ for all $n\geq 2$, we obtain the expected result.
\end{proof}

We now deal with the proof of Lemma \ref{lemme: BDGNonCoAdapte}.
\begin{proof}
[Proof of Lemma \ref{lemme: BDGNonCoAdapte}:]
 By construction of the coupling, it is enough to prove these results for $i=1$. To simplify the notations we also suppose that the Brownian motions start from the same fiber, i.e., $\tau_0=0$. 
For all $t\in[0,\tau_1]$, we have:
\begin{equation*}
    \|Y_1(t)\|_2^4\leq n\sum\limits_{j=1}^n|X^j(t)-X^j(0)|^4
    \text{ and } \|V_1(t)\|_2^2\leq \frac{1}{2}\sum\limits_{1\leq j\neq l\leq n}\left(\int_0^t(X_s^j-X_0^j)dX^l_s\right)^2.
\end{equation*}
Thus, the lemma will be proven if we can show that there exists a constant $\tilde{C}$ (not depending on the starting points of the processes or the rank $n$), such that, for any $1\leq j\neq l\leq n$:
\begin{equation}\label{eq: generalisationLemme}
     \esp\left[\mathbb{1}_{\{\tau_1<M\}}\sup\limits_{0\leq t \leq\tau_1}\left(X^j(t)-X^j(0)\right)^4\right]\leq \tilde{C}\esp\left[\left(\tau_1\wedge M\right)^2 \right]
     \end{equation}
     and
   \begin{equation}\label{eq: generalisationLemme2}
   \esp\left[\mathbb{1}_{\{\tau_1<M\}}\sup\limits_{0\leq t \leq\tau_1}\left(\int_0^t \left(X^j_s-X^j_0\right) dX^l(s)\right)^2\right]\leq \tilde{C}\esp\left[\left(\tau_1\wedge M\right)^2 \right].
\end{equation}
These two results can be seen as a generalisation of \emph{Lemma 4.1} from \cite{banerjee2017coupling} to higher dimensions. Although the method is quite similar, it needs some adjustments due to the change of the dimension. {We give here the details of the proof for the convenience of the reader and will point out the major elements needing extra care.}

As in the case of the Heisenberg group, we want to use the B-D-G Inequality. The main difficulty is due to the fact that the coupling is not co-adapted. Thus, $\tau_1$ is not a stopping time for the usual filtration induced by $(X_s)_{s}$ and we need to consider an enlarged filtration. The second difficulty comes from the fact that the considered processes (or most of them) have to be martingales for this new filtration.

Using the same notations as in the proof of Proposition \ref{Brique1}, we define:
     \begin{align}
         \chi_2(t)&:=\sum\limits_{k\geq 0}\mathbb{1}_{\{t_k\leq t\}}\left(\sum\limits_{m\geq n}\xi^{(k)}_m\frac{\sqrt{2T_k}}{m\pi}\sin\left(\frac{m\pi \left((t-t_k)\wedge T_k\right)}{T_k}\right)+\xi_0^{(k)}\frac{\left((t-t_k)\wedge T_k\right)}{\sqrt{T_k}}\right);\\
        \text{ and } \chi_1(t)&:=\sum\limits_{k\geq 0}\mathbb{1}_{\{t_k\leq t\leq t_{k+1}\}}\sum\limits_{m=1}^{n-1}\xi^{(k)}_m\frac{\sqrt{2T_k}}{m\pi}\sin\left(\frac{m\pi (t-t_k)}{T_k}\right)
     \end{align}
     such that $X_t^1=X_0^1+\chi_1(t)+\chi_2(t)$ for all $t\in[0,\tau_1]$.
    
     We {remind the reader} that, for all $k\geq 0$, $\left(\xi_m^{(k)}\right)_{1\leq m\leq n-1}=\sum\limits_{l=1}^{n-1} W_1^{l,(k)} f_l^{(k)}$ with $W^{(k)}$ a $\mathbb{R}^{n-1}$-Brownian motion independent of the basis $\left(f_l^{(k)}\right)_{1\leq l\leq n-1}$. In particular, $\left(\xi_m^{(k)}\right)_{1\leq m\leq n-1}$ has the distribution $\mathcal{N}(0,I_{n-1})$ because of this independence. With the same argument, we can define a Brownian motion $(C_t)_{0\leq t\leq \tau_1}$ starting from $0$ by taking $C_t:=\sum\limits_{k\geq 0}\sqrt{T_k}\sum\limits_{l=1}^{n-1}\mathbb{1}_{\{t\geq t_k\}} W_{\frac{(t-t_k)\wedge T_k}{T_k}}^{l,(k)} f_l^{(k)}$. In particular $(C_t)_{0\leq t\leq \tau_1}$ is independent of all the bases $\left(f_l^{(k)}\right)_{1\leq l\leq n-1}$. We also have $\left(\xi_m^{(k)}\right)_{1\leq m\leq n-1}=\frac{C_{t_{k+1}}-C_{t_k}}{\sqrt{T_k}}$. For $t\geq 0$, we define two filtrations:
     \begin{align}\label{eq: filtration1}
         \mathcal{F}_t^{\star}:=\sigma\Big(\{X_s^j,s\leq t\wedge\tau_1,2\leq j\leq n\}&\cup\{C_s,0\leq s\leq \tau_1\}\notag\\
         &\cup\{\xi_0^{(k)},k\geq 0\}\cup\{\xi_m^{(k)}, m\geq n, k\geq 0\}\Big)
     \end{align} and
       \begin{align}\label{eq: filtration2}
         \mathcal{F}_t^{\star \star}:=\sigma\Big(\{X_s^j,s\leq t\wedge\tau_1,2\leq j\leq n\}&\cup\{C_s,0\leq s\leq t\wedge\tau_1\}\notag\\
         &\cup\{\xi_0^{(k)},k\geq 0\}\cup\{\xi_m^{(k)}, m\geq n, k\geq 0\}\Big).
     \end{align}
         Note here, that, contrary to the case of the Heisenberg group, our coupling strategy needs the introduction, for all $k\geq 0$, of the basis $\left(f_l^{(k)}\right)_{1\leq l\leq n-1}$ which depends on $(X_s^j)_{t_k\leq t \leq t_{k+1}}$ for all $2\leq j\leq n$. We then have to take extra care to define the filtration $(\mathcal{F}^{\star}_t)_t$ with $(C_{t\wedge\tau_1})_{t}$ and not with a direct concatenation of  $(W^{(k)}_s)_{0\leq s\leq 1}$. Then $(X_{t\wedge\tau_1}^1-X_0^1)_{t}$ is $\mathcal{F}_0^{\star}$-measurable which will be important in the future to deal with the quantity \\$\esp\left[\mathbb{1}_{\{\tau_1\leq M\}}\sup\limits_{0\leq t \leq\tau_1}\left(\int_0^t \left(X^1_s-X^1_0\right) dX^l_s\right)^2\right]$ with $2\leq l\leq n$ (upcoming inequality (\ref{eq: BDGMartX^1})).
         
    {Since $\tau_1$ only takes its values in $\{t_k,\ k\geq 0\}$, with the same arguments as in the proof in~\cite{banerjee2017coupling},}
    $\tau_1$ is a stopping time for the filtrations $(\mathcal{F}_t^{\star \star})_{t}$ and $(\mathcal{F}_t^{\star})_{t}$. 
    
    For $2\leq j\leq n$, $(X_{t\wedge\tau_1}^j)_{t}$ is clearly a (stopped) Brownian motion adapted to the two filtrations $(\mathcal{F}_t^{\star \star})_t$ and $(\mathcal{F}_t^{\star})_t$. Thus, using the B-D-G Inequality, we obtain (\ref{eq: generalisationLemme}):
    \begin{equation}\label{eq: BDGXj}
         \esp\left[\sup\limits_{0\leq t \leq\tau_1\wedge M}\left(X^j(t)-X^j(0)\right)^4\right]\leq c_4\esp\left[\left(\tau_1\wedge M\right)^2 \right],
    \end{equation}
    with $c_4$ the universal constant for the B-D-G Inequality.
    With the same arguments, {using the B-D-G Inequality twice}, we obtain (\ref{eq: generalisationLemme2}) for $2\leq j\neq l\leq n$:
    \begin{align}\label{eq: BDGMart}
   \esp\left[\sup\limits_{0\leq t \leq\tau_1\wedge M}\left(\int_0^t \left(X^j_s-X^j_0\right) dX^l_s\right)^2\right]
   &\leq c_2 \esp\left[ \int_0^{\tau_1\wedge M} \left(X^j_s-X^j_0\right)^2 ds\right]\notag \\
   &\leq c_2 \esp\left[\left(\tau_1\wedge M \right) \sup\limits_{0\leq t \leq\tau_1\wedge M}\left(X^j_s-X^j_0\right)^2  \right]\notag \\
   &\leq c_2\esp\left[\left(\tau_1\wedge M\right)^2 \right]^{\frac{1}{2}}\esp\left[\sup\limits_{0\leq t \leq\tau_1\wedge M}\left(X^j_s-X^j_0\right)^4  \right]^{\frac{1}{2}}\notag\\
   &\leq c_2\sqrt{c_4}\esp\left[\left(\tau_1\wedge M\right)^2 \right].
\end{align}
Again, $c_2$ and $c_4$ are the two universal constants intervening in the B-D-G Inequality. Note that these universal constant can be estimated (see~\cite{BDGestimate}).

We now look for (\ref{eq: generalisationLemme}) for $j=1$. Let $k_M\geq 0$ such that $M\in[t_{k_M},t_{k_M+1}[$. Let $t_{k+1}\leq \tau_1\wedge t_{k_M}$ and $t\in[t_k,t_{k+1}]$. We denote by $(e_1,\hdots,e_{n-1})$ the {canonical} basis in $\mathbb{R}^{n-1}$. We get:
\begin{align*}
    |\chi_1(t)|^2&= \left|\sum\limits_{m=1}^{n-1}\langle C_{t_{k+1}}-C_{t_k},e_m\rangle\frac{\sqrt{2}}{m\pi}\sin\left(\frac{m\pi (t-t_k)}{T_k}\right)\right|^2\leq \left\| C_{t_{k+1}}-C_{t_k}\right\|_2^2 \frac{2}{\pi^2}\sum\limits_{m=1}^{n-1}\frac{1}{m^2}\notag\\
    &
    \leq\frac{4}{3}\sup\limits_{t_k\leq t\leq t_{k+1}}\left\| C_{t}\right\|_2^2\leq \frac{4}{3}\sup\limits_{0\leq t\leq \tau_1\wedge M}\left\| C_{t}\right\|_2^2.
\end{align*}
As $\tau_1$ takes its values in $\{t_k,k\geq 0\}$, there exists $0\leq k\leq k_M$ such that $\tau_1\wedge t_{k_M}=t_k$. Thus we have:
\begin{equation*}\sup\limits_{0\leq t\leq\tau_1\wedge t_{k_M}}|\chi_1(t)|^2\leq \frac{4}{3}\sup\limits_{0\leq t\leq \tau_1\wedge M}\left\| C_{t}\right\|_2^2.
\end{equation*}
Since $(C_{t\wedge\tau_1})_{t}$ is a (stopped) Brownian motion adapted to the filtration $(\mathcal{F}_t^{\star \star})_{t}$, we can use the B-D-G Inequality to obtain: 
\begin{equation}\label{eq: BDG chi_1}
\esp\left[\sup\limits_{0\leq t\leq\tau_1\wedge t_{k_M}}|\chi_1(t)|^4\right]\leq \frac{16}{9}\left[\sup\limits_{0\leq t\leq \tau_1\wedge M}\left\| C_{t}\right\|^4\right]\leq \frac{16c_4}{9}\esp\left[\left(\tau_1\wedge M\right)^2\right].
\end{equation}

We now deal with $\esp\left[\sup\limits_{0\leq t\leq\tau_1\wedge t_{k_M}}|\chi_2(t)|^4\right]$. As in \cite{banerjee2017coupling}, we remark that $\chi_2$ is independent of $\tau_1$. Conditioning by $\tau_1$, we get: 
\begin{align}\label{eq: BDG chi_2}
\esp\left[\sup\limits_{0\leq t\leq\tau_1\wedge t_{k_M}}|\chi_2(t)|^4\right]&\leq 8\esp\left[\sup\limits_{0\leq t\leq \tau_1\wedge t_{k_M}}|\chi_1(t)|^4\right]+8\esp\left[\esp\left[\left|\sup\limits_{0\leq t\leq \tau_1\wedge t_{k_M}}|X^1_t-X^1_0|^4\right|\tau_1\right]\right]\notag\\
&\leq 8\esp\left[\sup\limits_{0\leq t\leq\tau_1\wedge t_{k_M}}|\chi_1(t)|^4\right]+8c_4\esp[(\tau_1\wedge t_{k_M})^2]\notag\\
&\leq 8c_4\frac{25}{9}\esp\left[\left(\tau_1\wedge M\right)^2\right].
\end{align}

From (\ref{eq: BDG chi_1}) and (\ref{eq: BDG chi_2}), we obtain (\ref{eq: generalisationLemme}) for $j=1$:
\begin{align}\label{eq: sumChi}
    \esp\left[\mathbb{1}_{\{\tau_1\leq M\}}\times\sup\limits_{0\leq t\leq\tau_1}|X^1_t-X^1_0|^4\right]&=\esp\left[\mathbb{1}_{\{\tau_1\leq t_{k_M}\}}\times\sup\limits_{0\leq t\leq\tau_1}|X^1_t-X^1_0|^4\right]\notag\\
    &\leq \esp\left[\sup\limits_{0\leq t\leq\tau_1\wedge t_{k_M}}|\chi_1(t)+\chi_2(t)|^4\right]\leq 3\times 8^2c_4 \esp\left[\left(\tau_1\wedge M\right)^2\right].
\end{align}

We can now focus on (\ref{eq: generalisationLemme2}) for $j=1$ and $2\leq l\leq n$. For the filtration $(\mathcal{F}^{\star}_t)_{t}$, $(X_{t\wedge\tau_1}^l)_{t}$ is an adapted Brownian motion. Moreover $\left(X_{t\wedge\tau_1}^1-X_0^1\right)_{t}$ is $\mathcal{F}^{\star}_0$-measurable. Then, using the B-D-G inequality and (\ref{eq: sumChi}):
\begin{align}\label{eq: BDGMartX^1}
    \esp&\left[\mathbb{1}_{\{\tau_1\leq M\}}\sup\limits_{0\leq t \leq\tau_1}\left(\int_0^t \left(X^1_s-X^1_0\right) dX^l_s\right)^2\right]\leq\esp\left[\sup\limits_{0\leq t \leq\tau_1\wedge t_{k_M}}\left(\int_0^t \left(X^1_s-X^1_0\right) dX^l_s\right)^2\right]\notag\\
    &\leq c_2\esp\left[\left(\tau_1\wedge t_{k_M}\right)^2 \right]^{\frac{1}{2}}\esp\left[\sup\limits_{0\leq t \leq\tau_1\wedge t_{k_M}}\left(X^1_t-X^1_0\right)^4  \right]^{\frac{1}{2}}\notag\\
   &\leq 8c_2\sqrt{3c_4}\esp\left[\left(\tau_1\wedge M\right)^2 \right].
\end{align}
Finally, to obtain (\ref{eq: generalisationLemme2}) with $j>2$ and $l=1$, we can use the Itô formula to get \begin{equation*}\int_0^t \left(X_s^j-X_0^j\right) dX^1(s)=(X_t^1-X_0^1)(X_t^j-X_0^j)-\int_0^t \left(X^1_s-X^1_0\right) dX^j(s).\end{equation*}
Using the previous results, we then have:
\begin{equation*}
    \esp\left[\mathbb{1}_{\{\tau_1\leq M\}}\sup\limits_{0\leq t \leq\tau_1}\left(\int_0^t \left(X^j_s-X^j_0\right) dX^1(s)\right)^2\right]\leq 16\sqrt{3c_4}( c_2+\sqrt{c_4})\esp\left[\left(\tau_1\wedge M\right)^2 \right].
\end{equation*}
\end{proof}
\section{Gradient estimates}\label{Sec: gradientEstimates}
In this section, we present several gradient estimates. The first ones gives estimates for the heat semi-group and can be directly deduced from the coupling rate obtained in Section \ref{Section:Brownian BridgeGn}. The second ones are about harmonic functions and come from the comparison between the first coupling time and the first exit time from a domain dealt with in Section \ref{sec: coupling Vs Exit}.
\subsection{Norms of gradients}\label{subsec: NormsGardientGn}
We first begin with some preliminaries on the horizontal and vertical gradients on the homogeneous Carnot groups. Using Remark \ref{NB: isomorphisme Gn}, the following objects and properties will keep sense on $\Ge_n$ for $n\geq 2$.

 Let $\Ge=(\mathbb{R}^{n+m},\circ)$ be a homogeneous Carnot group with step $2$ and rank $n$. For any $f$ smooth enough, we define the horizontal gradient by \begin{equation*}\nabla_{\mathcal H} f:=\sum_{i=1}^n \bX_i(f) \bX_i.\end{equation*}
 Denoting by 
$\|\cdot  \|_{\mathcal H}$ the Euclidean norm on $\mathcal{H}$, for any $g\in \Ge$, we have:
 \begin{equation*}||\nabla_{\mathcal{H}}f(g)||_{\mathcal{H}}:=\sqrt{\sum\limits_{i=1}^n(\bar{X_i}f)^2(g)}.\end{equation*} 
 To obtain estimates of this norm, we use upper gradients. Let us first recall the definition of an upper gradient. We say that a function $u$ on $\Ge$ is an upper gradient of $f$ if, for every horizontal curve $\gamma:[0,T]\rightarrow \Ge$ parameterized with the arc-length, we have:
\begin{equation*}
|f(\gamma(0))-f(\gamma(t))|\leq \int_0^tu(\gamma(s))ds.
\end{equation*}
In particular, $||\nabla_{\mathcal{H}}\cdot||_{\mathcal{H}}$ is an upper-gradient.

As $\Ge$ has a left-invariant structure, it is a regular subRiemannian manifold as described in~\cite{HuangSublaplacian}, and, 
for all $u$ upper gradient of $f$:
\begin{equation}\label{eq: upper-gradientInequality}
||\nabla_{\mathcal{H}}f(g)||_{\mathcal{H}}\leq u(g) \text{ a.e. in }E_k.
\end{equation}
See~\cite{SobolevMetPoincarre,HuangSublaplacian} for some proofs and more details.
In particular, for $f$ Lipschitz, an upper gradient will be given by the gradient length associated to the metric space $(\Ge,d_{cc})$ (see~\cite{kuwada,SobolevMetPoincarre}):
\begin{equation}\label{gradient length}|\nabla f|(g):=\lim\limits_{r\downarrow 0}\sup\limits_{\substack{
		g\neq \tg\\
		d_{cc}(g,\tg)<r}
}\frac{\lvert f(g)-f(\tg)\rvert}{d_{cc(g,\tg)}}.
\end{equation}
Thus, using (\ref{eq: upper-gradientInequality}), we have:
\begin{equation}\label{eq:Uppergradient}
||\nabla_{\mathcal{H}}f(g)||_{\mathcal{H}}\leq |\nabla f|(g) \text{ a.e. in }\Ge.
\end{equation}
When $\Ge\approx \Ge_n$, we can consider the gradient length associated to the pseudo-metric $\bd$: 
\begin{equation*}|\nabla f|_{\bd}(g):=\lim\limits_{r\downarrow 0}\sup\limits_{\substack{
		g\neq \tg\\
		\bd(g,\tg)<r}
}\frac{\lvert f(g)-f(\tg)\rvert}{\bd(g,\tg)} \text{ for any function }f\text{ Lipschitz on }(\Ge,d_{cc}).
\end{equation*}
By equivalence between the pseudo-metric $\bd$ and the Carnot-Carathéodory metric $d_{cc}$, this notation makes sense. In particular, as $\bd(g,\tg)\leq \frac{1}{{m_1(n)}}
d_{cc}(g,\tg)$ (see \eqref{eq: equivSeminorme}), we have \begin{equation}\label{comparaisonLongueurGradients}|\nabla P_t f|_{\bd}(g)\geq {m_1(n)}
|\nabla P_t f|.
\end{equation}


We can also define the vertical gradient: \begin{equation*}\nabla_{\frak v} f= \sum_{k=1}^m \bar{Z}_{k} (f) \bar Z_{k}.\end{equation*}
Considering the Euclidean metric on each fiber $G_x:=\{(x,z)\in\Ge \ |\ z\in\mathbb{R}^m\}\approx\mathbb{R}^m$, as for the horizontal gradient, we can evaluate $\|\nabla_{\frak v} f(g)\|=\sqrt{\sum\limits_{k=1}^m \bZ_k(f)^2(g)}$. By using the gradient length associated to $G_x$ endowed with the Euclidean norm, for $g=(x,z)$ and $f$ Lipschitz, we have:
\begin{equation}\label{eq: VertUppergradient}
\|\nabla_{\frak v} f(x,z)\|\leq\lim\limits_{r\downarrow 0}\sup\limits_{\substack{
		z\neq \tz\\
		\|z-\tz\|_2<r}
}\frac{\lvert f(x,z)-f(x,\tz)\rvert}{\|z-\tz\|_2} a.e..
\end{equation}
Note that, when $\Ge\approx\Ge_n$, \eqref{eq: VertUppergradient} is still true for $z\in\mathfrak{so}(n)$ by definition of $\|\cdot\|_2$ (see \eqref{NB: isomorphisme Gn}).
\subsection{Estimates of the horizontal and vertical gradient for the heat semi-group}\label{Subsec: SuccesHorizontalVert}
	We first obtain an upper-bound (with an explicit estimate) for the horizontal gradient of the heat semi-group for any homogeneous step $2$ Carnot groups.
Let $g\in\Ge$ and $(\bbb_t)_t$ be a Brownian motion starting at $g$. We recall that, for all measurable bounded function $f$ on $\Ge$, $P_tf(g)=\esp[f(\bbb_t)]$. 
\begin{cor}\label{cor: horGradGnSucces}
    	Let $\Ge$ be a homogeneous Carnot group of step $2$ and of rank $n\geq 2$. For any bounded measurable function $f$ on $\Ge$, for any $g\in \Ge$ and $t\geq 1$:
	\begin{equation}\label{gradientInequality2}
	||\nabla_{\mathcal{H}}P_tf||_{\mathcal{H}}\leq2||f||_{\infty}\frac{C_1(n)}{\sqrt{t}}\text{ a.e.}
	\end{equation}
	with $C_1(n)$ the explicit constant from Theorem \ref{thm: successfulGn}.
\end{cor}
\begin{proof}
We consider $f$ a bounded measurable function on $\Ge$ and $g,\tg\in \Ge$. We define the following coupling $(\bbb_t,\tbbb_t)_t$:
\begin{itemize}
    \item still denoting by $\tau$ the first coupling time, for $t\in[0,\tau]$, we take $(\bbb_t,\tbbb_t)_{0\leq t\leq\tau}$ the coupling constructed in Corollary \ref{cor: couplingAllHomogeneous};
    \item for $t\geq \tau$, we take $\bbb_t=\tbbb_t$.
\end{itemize}
We have: 
\begin{align}\label{inegalité}
	|P_tf(g)-P_tf(\tg)|&=|\esp[f(\bbb_t)-f(\tbbb_t)]|\leq\esp[|f(\bbb_t)-f(\tbbb_t)|]\notag\\
	&= \esp[|f(\bbb_t)-f(\tbbb_t)|\mathbb{1}_{\{{\tau}>t\}}]\notag\\
	&\leq 2||f||_{\infty}\pr(\tau>t).
	\end{align}
	Using 
	the estimate of the coupling rate (\ref{eq: couplingRateAllHomogeneous}), we get, for $t$ {not too close} to $0$:
	\begin{equation}\label{test}
	|P_tf(g)-P_tf(\tg)|\leq 2\|f\|_{\infty}\left(C_1(n)\frac{d_{cc}(g,\tg)}{\sqrt{t}}\mathbb{1}_{x\neq\tx}+\frac{C_2(n)}{{m_1(n)}^2}\frac{d_{cc}(g,\tg)^2}{t}\right).
	\end{equation}
	In particular, for $t$ {not too close} to $0$ and $f$ bounded measurable, $g\mapsto P_tf(g)$ is Lipschitz on $(\Ge,d_{cc})$.
     We have:
    \begin{equation}|\nabla P_tf|(g)=\lim\limits_{r\downarrow 0}\sup\limits_{\substack{
		g\neq \tg\\
		{d_{cc}}(g,\tg)<r}
}\left\lvert\frac{P_tf(g)-P_tf(\tg)}{d_{cc}(g,\tg)}\right\rvert\leq 2||f||_{\infty}\frac{C_1(n)}{\sqrt{t}}.\end{equation}
 As the gradient length $|\nabla P_tf|$ is an upper-gradient of $P_tf$, from Inequality (\ref{eq: upper-gradientInequality}), we obtain the expected result.
\end{proof}
We can also obtain estimates on $\Ge_n$ for the vertical gradient as defined in Subsection \ref{subsec: NormsGardientGn}.
\begin{cor}\label{cor: vertGradGnSucces}
    	Let $n\geq 3$. For any bounded measurable function $f$ on $\Ge_n$ and for any $t\geq 1$:
	\begin{equation}\label{gradientInequality4}
	\|\nabla_{\mathfrak{v}}P_t f\|\leq||f||_{\infty}\frac{C_2(n)}{t}\text{ a.e.}
	\end{equation}
	with $C_2(n)$ the explicit constant from Theorem \ref{thm: successfulGn}.
\end{cor}
\begin{proof}
  Denote $g=(x,z)\in\Ge_n$. We set $\tz\in\mathfrak{so}(n)$ and $\tg=(x,\tz)\in\Ge_n$. Applying (\ref{inegalité}) together with the coupling rate estimate (\ref{eq: InegaliteMemeFibre}), we get:
	\begin{equation*}
	|P_tf(g)-P_tf(\tg)|\leq 2\|f\|_{\infty}\beta_n\frac{\|z-\tz\|_2}{t}\text{ with }\beta_n=\frac{C_2(n)}{2}.
	\end{equation*}
	As previously, we obtain an inequality for the gradient length defined on the fiber $\{(x,z)\in\Ge_n\ | \ z\in\mathfrak{so}(n)\}$. We use Inequality (\ref{eq: VertUppergradient}) to obtain the expected result.
\end{proof}

\subsection{Estimates of the gradient for the harmonic functions on $\Ge_n$}\label{Subsec GradHarmonique}
By using the results from Theorem \ref{thm: sortieDomaine}, we can obtain the same results than in \emph{Corollaries 4.4, 4.5 and 4.6} from~\cite{banerjee2017coupling}. As the proofs are exactly the same, we just give here an overview of the method and refer the reader to~\cite{banerjee2017coupling} for more details.

As in Section \ref{sec: coupling Vs Exit}, we consider a domain $D$ on $\Ge_n$ and $\bar{D}$ its closure. Let $f$ be a continuous function on $\bar{D}$, smooth and such that $f$ is harmonic on $D$, i.e., $Lf(g)=0$ for $g\in D$. For any set $A\subset D$, we define $\osc_A(f):=\sup\limits_A f-\inf\limits_A f$. In particular, supposing $\bar{D}$ compact, by continuity of $f$ on $\bar{D}$, $\osc_{D}(f)$ is finite. Let $g=(x,z),\tg=(\tx,\tz)\in D$, as $f$ is harmonic, by the Itô formula, for any successful coupling $(\bbb_t,\tbbb_t)_t$ (defined such that $\bbb_t=\tbbb_t$ after the first coupling time $\tau$ as previously) we get:
\begin{align*}
    |f(g)-f(\tg)|&=\left|\esp\left[f(\bbb_{\tau_D(\bbb)})-f(\tbbb_{\tau_D(\tbbb)})\right]\right|\\
    &\leq \osc_D(f)\pr\left(\tau>\tau_D(\bbb)\wedge\tau_D(\tbbb)\right)
\end{align*}
In particular, using the results from Theorem \ref{thm: sortieDomaine}, for $\tg$ close enough to $g$, we get:
\begin{align*}
    |f(g)-f(\tg)|&\leq \osc_D(f)\left({D_1}\times\left(n\sqrt{n}+\frac{n\sqrt{n}}{\delta_g}+\frac{{n^{10+\frac{1}{4}}}}{\delta_g^4}\right)\|x-\tx\|_2\right.\\
    &\left.+{D_2}\times \left(1+\frac{{n^{9+\frac{1}{2}}}}{\delta_g^4}\right)\max\left(
    \|\zeta\|_2,\|\zeta\|_2^2\right)\right).
\end{align*}
Finally, using the gradient length associated to the pseudo-distance $\bd$, we get: 
\begin{equation}|\nabla P_tf|_{\bd}(g)\leq \osc_D(f) {D_1}\times\left(n\sqrt{n}+\frac{n\sqrt{n}}{\delta_g}+\frac{{n^{10+\frac{1}{4}}}}{\delta_g^4}\right).\end{equation}
Using \eqref{comparaisonLongueurGradients}, we have: ${m_1(n)}
|\nabla P_t f|\leq |\nabla P_t f|_{\bd}(g) $. Then we obtain:
\begin{cor}\label{cor: osculateur1}
   Let $n\geq 2$ and $D$ a domain on $\Ge_n$ such that its closure $\bar{D}$ is compact. For any $f\in \mathcal{C}(\bar{D})\cap\mathcal{C}^2(D)$ harmonic on $D$:
   \begin{equation*}||\nabla_{\mathcal{H}}f(g)||_{\mathcal{H}}\leq\osc_D(f) {\frac{1}{m_1(n)}}D_1
   \left(n\sqrt{n}+\frac{n\sqrt{n}}{\delta_g}+\frac{{n^{10+\frac{1}{4}}}}{\delta_g^4}\right)\text{ a.e. on }D\end{equation*}
   with $D_1$ the same constant as in Theorem \ref{thm: sortieDomaine} which depends neither on $n$, nor on $D$.
\end{cor}
 We now consider the result from Corollary \ref{cor: osculateur1} when we take $D=B_{\Ge_n}(g,r)$ with $r>0$ and $B_{\Ge_n}(g,r)$ the usual open ball for the Carnot Carathéodory distance $d_{cc}$. Using the homogeneous invariant Harnack Inequality (see for example \emph{Theorem 5.7.2} from~\cite{bonfiglioli2007stratified}), there exists a constant $\bar{C}$ that does not depend on $g$ and $r$ but only on the metric $d_{cc}$ such that $\osc_{B_{\Ge_n}(g,r)}(f)\leq \bar{C}f(g)$.
This leads to the following result:
\begin{cor}\label{cor: osculateur2}
   Let $n\geq 2$. There exists a constant $C$ depending on $\Ge_n$ and $d_{cc}$ such that, for all $g\in \Ge_n$, $r>0$ and for any smooth non-negative function $f\in\mathcal{C}(\bar{B}_{\Ge_n}(g,r))$ harmonic on $B_{\Ge_n}(g,r)$ we have:
   \begin{equation*}||\nabla_{\mathcal{H}}f(g)||_{\mathcal{H}}\leq C\left(n\sqrt{n}+\frac{n\sqrt{n}}{\delta_g}+\frac{{n^{10+\frac{1}{4}}}}{\delta_g^4}\right)f(g)\text{ a.e..}\end{equation*}
\end{cor}
Finally, as $\frac{||\nabla_{\mathcal{H}}f(g)||_{\mathcal{H}}}{f(g)}=||\nabla_{\mathcal{H}}\log f(g)||_{\mathcal{H}}$, we can obtain the following Cheng-Yau estimate:
\begin{cor}\label{cor: cheng-Yau}
   Let $n\geq 2$ and $g_0\in \Ge_n$. We consider a real $r>0$ and a smooth non-negative function $f$ harmonic on $B_{\Ge_n}(g_0,2r)$. Then:
   {\begin{equation*}
   \sup\limits_{g\in B_{\Ge_n}(g_0,r)}||\nabla_{\mathcal{H}}\log f(g)||_{\mathcal{H}}\leq \frac{C}{r}\text{ a.e.}\end{equation*}}
   with $C$ a constant that does depend only on $\Ge_n$ and $d_{cc}$ (neither on $g_0$ nor on $r$).
\end{cor}

\subsection{Estimates of the gradient for the harmonic functions on all homogeneous step $2$ Carnot groups}\label{Subsec GradHarmoniqueCarnot}
We can also extend the results from the previous section to all homogeneous, step $2$ Carnot groups $\Ge$. 

Set $n\geq 2$ the rank of $\Ge$. For $r>0$, $g\in \Ge$, we denote $B_{\Ge}(g,r)$ the open ball for the Carnot Carathéodory metric. 
Consider the surjective morphism $\phi$ from Proposition \ref{epimorphisme} and $a\in\phi^{-1}(\{g\})$. We can show that  $\phi\left(B_{\Ge_n}(a,r)\right)=B_{\Ge}(g,r)$.
Let $\tg\in B_{\Ge}(g,r)$. There exists $\tilde a\in \Ge_n$ such that $d_{cc}(a,\tilde a)=d_{cc}(g,\tg)$. In particular, $\tilde a\in B_{\Ge_n}(a,r)$. As in Corollary $\ref{cor: couplingAllHomogeneous}$, we can construct a coupling $(\mathbf B_t,\tilde{\mathbf{B}}_t)_t$ of Brownian motion starting at $(a,\tilde a)$ and a coupling $(\bbb_t,\tbbb_t)_t:=\left(\phi(\mathbf B_t),\phi(\tilde{\mathbf{B}}_t)\right)_t$ of Brownian motions on $\Ge$ starting at $(g,\tg)$.
We also choose $\tilde{\mathbf{B}}_t={\mathbf{B}}_t$ for $t\geq \tau:=\inf\{s\geq 0\,|\,\tilde{\mathbf{B}}_s={\mathbf{B}}_s\}$.

Set $f$ a smooth function, harmonic on $B_{\Ge}(g,r)$ and continuous on $\bar{B}_{\Ge}(g,r)$. By construction of $\phi$, we also have $f\circ\phi$ smooth and harmonic on $B_{\Ge}(a,r)$ and continuous on $\bar{B}_{\Ge}(a,r)$. With the same method as for Corollary \ref{cor: osculateur1}: 
\begin{align*}
    |f(g)-f(\tg)|&=|f\circ\phi(a)-f\circ\phi(\ta)|\\
    &\leq \osc_{B_{\Ge}(a,r)}(f\circ \phi)\pr\left(\tau>\tau_{B_{\Ge}(a,r)}(\mathbf B)\wedge\tau_{B_{\Ge}(a,r)}(\tilde{\mathbf{B}})\right)\\
    &\leq \osc_{B_{\Ge}(g,r)}(f)\left({D_1}\times\left(n\sqrt{n}+\frac{n\sqrt{n}}{\delta_a}+\frac{n^{{10+\frac{1}{4}}}}{\delta_a^4}\right)d_{cc}(a,\tilde a)\right.\\
    &\left.+{D_2}\times \left(1+\frac{n^{{9}+\frac{1}{2}}}{\delta_a^4}\right)\max\left(
    \frac{d_{cc}(a,\tilde a)^2}{{m_1(n)^2}},\frac{d_{cc}(a,\tilde a)^4}{{m_1(n)^4}}\right)\right).
\end{align*}
In particular using the equivalence relation \eqref{eq: equivSeminorme}: 
\begin{align*}\bd_a=\bd\left(a,\bar{B}_{\Ge_n}(a,r)^{c}\right)\geq {\frac{r}{m_2(n)}}.
\end{align*}
As $d_{cc}(a,\tilde a)=d_{cc}(g,\tg)$:
\begin{align}
    ||\nabla_{\mathcal{H}}f(g)||_{\mathcal{H}}
    &\leq \osc_{B_{\Ge}(g,r)}(f){D_1}\left(n\sqrt{n}+{m_2(n)}\frac{n\sqrt{n}}{r}+{m_2(n)^4}\frac{n^{{10+\frac{1}{4}}}}{r^4}\right).
\end{align}
We can then extend Corollaries \ref{cor: osculateur1}, \ref{cor: osculateur2} and \ref{cor: cheng-Yau} to all homogeneous step $2$ Carnot groups.

	\bibliographystyle{plain}
	\bibliography{Bibliographie}
\end{document}